\documentclass[pdftex]{article} 

\usepackage{amsmath,amssymb,amsthm,slashbox}
\usepackage{tikz,url}

\title{Refining the arithmetical hierarchy of classical principles}
\author{Makoto Fujiwara\footnote{Email: makotofujiwara@rs.tus.ac.jp}
\footnote{Department of Applied Mathematics, Faculty of Science Division I, Tokyo University of Science, 1-3 Kagurazaka, Shinjuku-ku, Tokyo 162-8601, Japan.}
and Taishi Kurahashi\footnote{Email: kurahashi@people.kobe-u.ac.jp}
\footnote{Graduate School of System Informatics,
Kobe University,
1-1 Rokkodai, Nada, Kobe 657-8501, Japan.}}
\date{}

\theoremstyle{plain}
\newtheorem{thm}{Theorem}[section]

\newtheorem{prop}[thm]{Proposition}
\newtheorem{cor}[thm]{Corollary}
\newtheorem{fact}[thm]{Fact}
\newtheorem{prob}[thm]{Problem}

\theoremstyle{definition}
\newtheorem{defn}[thm]{Definition}

\newtheorem{remark}[thm]{Remark}

\newcommand{\PA}{\mathsf{PA}}
\newcommand{\HA}{\mathsf{HA}}

\newcommand{\n}[1]{#1^{\mathrm{n}}}
\newcommand{\dn}[1]{#1^{\mathrm{dn}}}
\newcommand{\FV}{\mathrm{FV}}
\newcommand{\Fml}{\mathrm{Fml}}

\newcommand{\DNE}[1]{#1\text{-}\mathbf{DNE}}
\newcommand{\DNS}[1]{#1\text{-}\mathbf{DNS}}
\newcommand{\LEM}[1]{#1\text{-}\mathbf{LEM}}
\newcommand{\DML}[1]{#1\text{-}\mathbf{DML}}
\newcommand{\LLPO}[1]{#1\text{-}\mathbf{LLPO}}
\newcommand{\DUAL}[1]{#1\text{-}\mathbf{DUAL}}
\newcommand{\WDUAL}[1]{#1\text{-}\mathbf{WDUAL}}
\newcommand{\CD}[1]{#1\text{-}\mathbf{CD}}
\newcommand{\COLL}[1]{#1\text{-}\mathbf{COLL^{cp}}}
\newcommand{\LN}[1]{#1\text{-}\mathbf{LN}}
\newcommand{\PEIRCE}[1]{#1\text{-}\mathbf{PEIRCE}}

\begin{document}

\maketitle

\begin{abstract}
We refine the arithmetical hierarchy of various classical principles by finely investigating the derivability relations between these principles over Heyting arithmetic. 
We mainly investigate some restricted versions of the law of excluded middle, de Morgan's law, the double negation elimination, the collection principle and the constant domain axiom. 
\end{abstract}

\section{Introduction}\label{section:intro}

The interrelations between weak logical principles over intuitionistic arithmetic have been studied extensively in these three decades (cf.~\cite{ABHK,F20,FIN,FuKa,FuKo,Ishi,Kohl}).
In particular, Akama et al.~\cite{ABHK} systematically studied the structure of the law of excluded middle $\mathbf{LEM}$ and the double negation elimination $\mathbf{DNE}$ restricted to prenex formulas and some related principles over intuitionistic first-order arithmetic $\HA$.
Interestingly, the derivability relation between them forms a beautiful hierarchy as presented in Figure \ref{fig:ABHK} (cf. \cite[Figure 2]{ABHK}).

\begin{figure}[h]
\centering
\begin{tikzpicture}
\node (S-LEM) at (0,0) {$\LEM{\Sigma_{k-1}}$};
\node (DLEM) at (0,1) {$\LEM{\Delta_k}$};
\node (PPDNE) at (2,2) {$\DNE{(\Pi_k \lor \Pi_k)}$};
\node (PLEM) at (2,3) {$\LEM{\Pi_k}$};
\node (SDNE) at (-2,2.5) {$\DNE{\Sigma_k}$};
\node (SLEM) at (0,4) {$\LEM{\Sigma_k}$};

\draw [-] (-1, -0.7)--(S-LEM);
\draw [-] (1, -0.7)--(S-LEM);
\draw [<-] (S-LEM)--(DLEM);
\draw [<-] (DLEM)--(PPDNE);
\draw [<-] (PPDNE)--(PLEM);
\draw [<-] (DLEM)--(SDNE);
\draw [<-] (PLEM)--(SLEM);
\draw [<-] (SDNE)--(SLEM);

\end{tikzpicture}
\caption{An arithmetical hierarchy of classical principles}\label{fig:ABHK}
\end{figure}
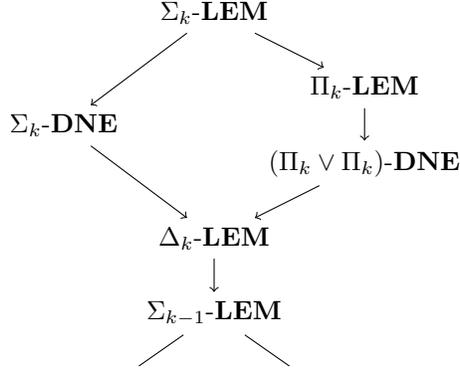

\noindent
By the prenex normal form theorem, which is first presented in \cite{ABHK} and corrected recently in \cite{FuKu}, this arithmetical hierarchy covers ${\bf LEM}$ for arbitrary formulas.
In this sense, the infinite hierarchy in Figure \ref{fig:ABHK} represents a gradual transition of strength of semi-classical arithmetic from $\HA$ to the classical arithmetic $\PA=\HA +{\bf LEM}$.
This hierarchy plays an important role in several aspects.
First, it is employed for the relativization of the relation between classical and intuitionistic arithmetic into the context of semi-classical arithmetic.
For example, $\PA$ is $\Pi_{k+2}$-conservative over $\HA+\LEM{\Sigma_k}$ for all natural numbers $k$ (see \cite[Section 6]{FuKu} and \cite{Ber04,FuKu21}).
In addition, for any theory $T$ in-between $\HA$ and $\PA$, the prenex normal form theorem for the classes of formulas $\mathrm{U}_{k'}$ (introduced in \cite{ABHK}) and $\Pi_{k'}$ holds in $T$ for all $k' \leq k$, if and only if, $T$ proves $\DNE{(\Pi_k \lor \Pi_k)}$ (see \cite[Section 7]{FuKu}).
Then the refinement of  the hierarchy is also important for analyzing the results on the relation between classical and intuitionistic arithmetic in more detail.
Secondly, the hierarchy is employed as a framework for a sort of constructive reverse mathematics over $\HA$ (cf.~\cite{BS14, BS17, Tof04}).
For example, Ramsey's theorem for pairs and recursive assignments of $2$ colors is located in the place of $\DNE{(\Pi_3 \lor \Pi_3)}$ (see \cite{BS14}).
Despite the fact that mathematical statements are usually not in prenex normal form, many of them are shown to be equivalent to some restricted logical principle in the arithmetical hierarchy (seemingly because the prenex normal form theorem is partly available in semi-classical arithmetic containing such logical principles).
Then the refinement of the hierarchy makes it possible to classify the logical strength of mathematical statements in finer classes.
After \cite{ABHK}, in connection with the development of constructive reverse mathematics \cite{Ishi05} over intuitionistic second-order arithmetic, further fine-grained analysis has been done for the principles with $k=1$ in the hierarchy (\cite{FIN, FuKo, Kohl}).
More recently, some connection between those principles and some other principles has been also found (\cite{F20,FuKa}).
Then it should be expected to recast the hierarchy in \cite{ABHK} based on these recent developments.
The history of the research of this line until \cite{FuKo} is summarized in \cite[Section 2.1]{FuKo}.

Motivated from them, we study the interrelations between various principles from the previous research and the related principles comprehensively in the context of $\HA$.
In particular, we investigate principles more finely and more systematically than ever before.
Such a fine-grained analysis reveals a more detailed hierarchical structure which the logical principles have.
In addition to the principles dealt with in \cite{ABHK}, we deal with de Morgan's law ${\bf DML}$, the (contrapositive) collection principle ${\bf COLL^{cp}}$ and the constant domain axiom ${\bf CD}$ systematically.
Among many other things, we show that $\DNE{(\Pi_k \lor \Pi_k)}$, $\DML{\Sigma_k}$ with respect to duals (which is $\Sigma_k$-${\bf LLPO}$ in \cite{ABHK}), $\DML{\Sigma_k} + \DNE{\Sigma_{k-1}}$, $\COLL{\Pi_k}$ and $\CD{(\Pi_k, \Pi_k)}$ are pairwise equivalent over $\HA$ for all natural numbers $k$ greater than $0$ (see Corollary \ref{cor:MC}).

The structure of the paper is as follows.
In Section \ref{section:dual}, we extract and investigate the principles concerning duals $\varphi^\bot$ (which are prenex formulas classically equivalent to $\neg \varphi$) of prenex formulas $\varphi$.
In Section \ref{section:LEM}, we investigate  variants of ${\bf LEM}$.
Section \ref{section:DML} is devoted to investigate several variations of ${\bf DML}$.
In particular, ${\bf LEM}$ for negated formulas is shown to be a variation of ${\bf DML}$. 
In Section \ref{section:DNE}, we investigate variants of ${\bf DNE}$.
In particular,  ${\bf DML}$ is shown to be a variation of ${\bf DNE}$.
Finally, we investigate ${\bf CD}$ in Section \ref{section:CD}. 
The results established in this paper are summarized in Section \ref{section:rem}, to which we refer the reader who merely wants to consult the results.

\section{Preliminaries}\label{section:pre}

In this paper, we work within the framework of first-order intuitionistic arithmetic with the logical connectives $\land, \lor, \to, \exists, \forall$ and $\bot$, where $\neg \varphi$ is the abbreviation of $\varphi \to \bot$. 
We may assume that the language of first-order arithmetic contains function symbols corresponding to all primitive recursive functions. 
Heyting arithmetic $\HA$ is an intuitionistic theory in the language of first-order arithmetic consisting of basic axioms for arithmetic, induction axiom scheme and axioms corresponding to defining equations of primitive recursive functions (see \cite[Section 3.2]{Kohl08}). 
Recall that $\varphi \to \neg \neg \varphi$, $(\varphi \to \psi) \to (\neg \psi \to \neg \varphi)$, $\neg \neg (\varphi \to \psi) \leftrightarrow (\neg \neg \varphi \to \neg \neg \psi)$, $\neg \neg \neg \varphi \to \neg \varphi$ and $\forall x\, \neg \varphi \leftrightarrow\, \neg\, \exists x \varphi$ etc.~are intuitionistically derivable. 
For more information about the logical implications over intuitionistic logic, we refer the reader to \cite[Section 6.2]{vD13}. 

Throughout this paper, we assume that $k$ always denotes a natural number $k \geq 0$. 
We define the family $\{\Sigma_k, \Pi_k : k \geq 0\}$ of sets of formulas inductively as follows: 
\begin{itemize}
	\item Let $\Sigma_0 = \Pi_0$ be the set of all quantifier-free formulas; 
	\item $\Sigma_{k+1} : = \{\exists x_1 \cdots \exists x_n \varphi \mid \varphi \in \Pi_k$, $n \geq 1$ and $x_1, \ldots, x_n$ are variables$\}$;  
	\item $\Pi_{k+1} : = \{\forall x_1 \cdots \forall x_n \varphi \mid \varphi \in \Sigma_k$, $n \geq 1$ and $x_1, \ldots, x_n$ are variables$\}$.  
\end{itemize}
For convenience, we assume that $\Sigma_m$ and $\Pi_m$ denote the empty set for any negative integer $m$. 
We say that a formula is in \textit{prenex normal form} if it is in $\Sigma_k$ or $\Pi_k$ for some $k$. 
Let $\FV(\varphi)$ denote the set of all free variables in $\varphi$. 
It is known that every formula $\varphi$ in $\Sigma_{k+1}$ (resp.~$\Pi_{k+1}$) is $\HA$-equivalent to a formula $\psi$ in $\Sigma_{k+1}$ (resp.~$\Pi_{k+1}$) such that $\FV(\varphi) = \FV(\psi)$ and $\psi$ is of the form $\exists x \psi'$ (resp.~$\forall x \psi'$) where $\psi'$ is $\Pi_k$ (resp.~$\Sigma_k$). 

Let $\Gamma$ and $\Theta$ be sets of formulas. 
We define $\Gamma \lor \Theta$, $\n{\Gamma}$ and $\dn{\Gamma}$ to be the sets $\{\varphi \lor \psi \mid \varphi \in \Gamma$ and $\psi \in \Theta\}$, $\{\neg \varphi \mid \varphi \in \Gamma\}$ and $\{\neg \neg \varphi \mid \varphi \in \Gamma\}$ of formulas, respectively. 
We adopt a convention that we write $\Gamma \subseteq \Theta$ if for any formula $\varphi \in \Gamma$, there exists a formula $\psi \in \Theta$ such that $\FV(\varphi) = \FV(\psi)$ and $\HA$ proves $\varphi \leftrightarrow \psi$. 
Then it is shown that $\Sigma_k \subseteq \Sigma_{k+1} \cap \Pi_{k+1}$ and $\Pi_k \subseteq \Sigma_{k+1} \cap \Pi_{k+1}$ (cf.~\cite{FuKu}). 

We introduce several principles which give semi-classical arithmetic as follows: 

\begin{defn}\label{defn:principles}
Let $\Gamma$ be any set of formulas. 
\begin{tabbing}
\hspace{25mm} \= \hspace{40mm} \= \hspace{50mm} \kill
$\LEM{\Gamma}$ \> $\varphi \lor \neg \varphi$ \> ($\varphi \in \Gamma$) \\
$\LEM{\Delta_k}$ \> $(\varphi \leftrightarrow \psi) \to \varphi \lor \neg \varphi$ \> ($\varphi \in \Sigma_k$ and $\psi \in \Pi_k$) \\
$\DNE{\Gamma}$ \> $\neg \neg \varphi \to \varphi$ \> ($\varphi \in \Gamma$) 
\end{tabbing}
\end{defn}

For each theory $T$ and principle $P$, let $T + P$ denote the theory obtained from $T$ by adding universal closures of all instances of $P$ as axioms. 
Since $\HA$ proves $\varphi \lor \neg \varphi \to (\neg \neg \varphi \to \varphi)$ for any formula $\varphi$, the following fact trivially holds. 

\begin{fact}\label{fact:LEM_DNE}
For any set $\Gamma$  of formulas, $\HA + \LEM{\Gamma} \vdash \DNE{\Gamma}$. 
\end{fact}

Nontrivial implications between the principles defined in Definition \ref{defn:principles} are investigated by Akama et al.~\cite{ABHK}. 
The following fact is visualized in Figure \ref{fig:ABHK} in Section \ref{section:intro}.  

\begin{fact}[Akama et al.~\cite{ABHK}]\label{fact:ABHK}\leavevmode
\begin{enumerate}
	\item $\LEM{\Sigma_k}$ and $\LEM{\Pi_k} + \DNE{\Sigma_k}$ are equivalent over $\HA$; 
	\item $\HA + \LEM{\Pi_k} \vdash \DNE{(\Pi_k \lor \Pi_k)}$; 
	\item $\HA + \DNE{(\Pi_k \lor \Pi_k)} \vdash \LEM{\Delta_k}$; 
	\item $\HA + \DNE{\Sigma_k} \vdash \LEM{\Delta_k}$;
	\item $\HA + \LEM{\Delta_{k+1}} \vdash \LEM{\Sigma_k}$; 
	\item $\DNE{\Sigma_k}$ and $\DNE{\Pi_{k+1}}$ are equivalent over $\HA$. 
\end{enumerate}
\end{fact}

In the present paper, we also deal with other important principles based on such as the double negation shift, de Morgan's law and the constant domain axiom. 

\begin{defn}\label{defn:principles2}
Let $\Gamma$ and $\Theta$ be any sets of formulas. 
\begin{tabbing}
\hspace{20mm} \= \hspace{55mm} \= \hspace{50mm} \kill
$\DNS{\Gamma}$ \> $\forall x\, \neg \neg \varphi(x) \to \neg \neg\, \forall x\varphi(x)$ \> ($\varphi(x) \in \Gamma$) \\
$\DML{\Gamma}$ \> $\neg (\varphi \land \psi) \to \neg \varphi \lor \neg \psi$ \> ($\varphi, \psi \in \Gamma$) \\
$\CD{(\Gamma, \Theta)}$ \> $\forall x (\varphi \lor \psi(x)) \to \varphi \lor \forall x \psi(x)$ \> ($\varphi \in \Gamma$, $\psi(x) \in \Theta$ and $x \notin \FV(\varphi)$) 
\end{tabbing}
\end{defn}

The principle $\DML{\Sigma_k}$ is introduced in \cite{BS14}. 
The principles defined in Definition \ref{defn:principles2} have mainly been investigated for $k=1$ in the literature. 
For example, $\DML{\Sigma_1}$ and $\DML{\Pi_1}$ correspond to the principle $\mathbf{LLPO}$ and disjunctive Markov's principle, respectively (see \cite{Ishi}). 
Also the principle $\LEM{\Delta_1}$ corresponds to the principle (IIIa) in \cite{FIN} and to the principle $\LEM{\Delta_a}$ in \cite{FuKo}. 
Notice that \cite{FIN,FuKa,Ishi} are studied in the context of second-order arithmetic. 
We have the following results from the proofs of the corresponding results in these papers. 

\begin{fact}[Ishihara {\cite[Proposition 1]{Ishi}}]\label{fact:Ishi}\leavevmode
\begin{enumerate}
	\item $\HA + \DNE{\Sigma_1} \vdash \DML{\Pi_1}$; 
	\item $\HA + \DML{\Sigma_1} \vdash \DML{\Pi_1}$. 
\end{enumerate}
\end{fact}

\begin{fact}[Fujiwara, Ishihara and Nemoto {\cite[Proposition 2]{FIN}}]\label{fact:FIN3}\leavevmode
$\HA + \DML{\Pi_1} \vdash \LEM{\Delta_1}$.  
\end{fact}

\begin{fact}[Fujiwara and Kawai {\cite[Proposition 4.2]{FuKa}}]\label{FuKaCD}
$\CD{(\Pi_1, \Pi_1)}$ and $\DML{\Sigma_1}$ are equivalent over $\HA$. 
\end{fact}

In the following sections, we investigate those principles more finely than ever before. 
In the process of the investigation, we also generalize the facts stated above. 

Concerning $\DNS{\Gamma}$, we easily obtain the following proposition. 

\begin{prop}\label{prop:DNS1}\leavevmode
\begin{enumerate}
	\item $\HA + \DNE{\Sigma_k} \vdash \DNS{\Sigma_k}$; 
	\item $\DNS{\Sigma_k}$ and $\DNS{\Pi_{k+1}}$ are equivalent over $\HA$.
\end{enumerate}
\end{prop}
\begin{proof}
1. Let $\varphi$ be any $\Sigma_k$ formula. 
Then $\HA + \DNE{\Sigma_k} \vdash \forall x\, \neg \neg \varphi \to \forall x \varphi$. 
We obtain $\HA + \DNE{\Sigma_k} \vdash \forall x\, \neg \neg \varphi \to \neg \neg\, \forall x \varphi$. 

2. We prove $\HA + \DNS{\Sigma_k} \vdash \DNS{\Pi_{k+1}}$. 
Let $\forall y \varphi(x, y)$ be any $\Pi_{k+1}$ formula where $\varphi(x, y) \in \Sigma_k$. 
Then $\HA \vdash \forall x\, \neg \neg\, \forall y \varphi(x, y) \to \forall x \forall y\, \neg \neg \varphi(x, y)$. 
Let $(z)_0$ and $(z)_1$ be primitive recursive inverse functions of a fixed pairing function which calculate the first and the second components of $z$ as a pair, respectively. 
Then $\HA \vdash \forall x\, \neg \neg\, \forall y \varphi(x, y) \to \forall z\, \neg \neg \varphi((z)_0, (z)_1)$. 
By applying $\DNS{\Sigma_k}$, we obtain $\HA + \DNS{\Sigma_k} \vdash \forall x\, \neg \neg\, \forall y \varphi(x, y) \to \neg \neg\, \forall z \varphi((z)_0, (z)_1)$.
We conclude $\HA + \DNS{\Sigma_k} \vdash \forall x\, \neg \neg\, \forall y \varphi(x, y) \to \neg \neg\, \forall x \forall y \varphi(x, y)$.
\end{proof}

A detailed investigation of the principle $\DNS{\Sigma_1}$ including Proposition \ref{prop:DNS1}.1 for $k=1$ is in \cite{FuKo}.

\section{The dual principles}\label{section:dual}

In \cite{FuKu}, the following result is proved. 

\begin{fact}[Fujiwara and Kurahashi {\cite[Lemma 4.7]{FuKu}}]\label{fact:DUAL}\leavevmode
\begin{enumerate}
	\item For any $\Sigma_k$ formula $\varphi$, there exists a $\Pi_k$ formula $\varphi'$ such that $\HA + \DNE{\Sigma_{k-1}} \vdash \neg \varphi \leftrightarrow \varphi'$; 
	\item For any $\Pi_k$ formula $\varphi$, there exists a $\Sigma_k$ formula $\varphi'$ such that $\HA + \DNE{\Sigma_k} \vdash \neg \varphi \leftrightarrow \varphi'$. 
\end{enumerate}
\end{fact}

In this section, we investigate the dual principles and the weak dual principles (see Definitions \ref{defn:DUAL} and \ref{defn:WDUAL}) motivated from Fact \ref{fact:DUAL}. 

\subsection{The dual principles}

First, we recall the notion of duals of formulas in prenex normal form, which is defined in \cite{ABHK} informally. 

\begin{defn}[cf.~\cite{ABHK}]\label{defn:DUAL}
For any formula $\varphi$ in prenex normal form, we define the dual $\varphi^\bot$ of $\varphi$ inductively as follows: 
\begin{enumerate}
	\item $\varphi^\bot : \equiv \neg \varphi$ if $\varphi$ is quantifier-free; 
	\item $(\forall x \varphi(x))^\bot : \equiv \exists x \varphi^\bot(x)$; 
	\item $(\exists x \varphi(x))^\bot : \equiv \forall x \varphi^\bot(x)$. 
\end{enumerate}
\end{defn}

The following proposition is a basic property of duals.  

\begin{prop}\label{prop:DUAL_BASIC}
Let $\varphi$ be any formula in prenex normal form. 
\begin{enumerate}
	\item If $\varphi$ is $\Sigma_k$ \textup{(}resp.~$\Pi_k$\textup{)}, then $\varphi^\bot$ is $\Pi_k$ \textup{(}resp.~$\Sigma_k$\textup{)}; 
	\item $\HA \vdash \varphi^{\bot \bot} \leftrightarrow \varphi$; 
	\item $\HA \vdash \varphi^\bot \to \neg \varphi$; 
	\item $\HA \vdash \neg (\varphi \land \varphi^\bot)$. 
\end{enumerate}
\end{prop}

\begin{proof}
1. Trivial. 

2. It is known that if $\varphi$ is $\Sigma_0$, then $\HA \vdash \neg \neg \varphi \leftrightarrow \varphi$. 
Then clause 2 is proved by induction on the number of quantifiers contained in $\varphi$. 

3. Notice that $\HA$ proves the formulas $\exists x\, \neg \varphi \to \neg\, \forall x \varphi$ and $\forall x\, \neg \varphi \to \neg\, \exists x \varphi$. 
Then clause 3 is also proved by induction on the number of quantifiers in $\varphi$. 

4. This is because $\HA \vdash \varphi \land \varphi^\bot \to \varphi \land \neg \varphi$ by clause 3. 
\end{proof}

From Propositions \ref{prop:DUAL_BASIC}.(1) and (2), we have that the mapping $(\cdot)^\bot$ is a bijection between $\Sigma_k$ (resp.~$\Pi_k$) and $\Pi_k$ (resp.~$\Sigma_k$) modulo $\HA$-provable equivalence. 

\begin{remark}
It is possible to extend the notion of duals in Definition \ref{defn:DUAL} (from \cite{ABHK}) to arbitrary formulas by the operation $(\cdot)^d$ defined inductively as
\begin{enumerate}
	\item $\varphi^d : \equiv \neg \varphi$ if $\varphi$ is prime; 
	\item
	$(\varphi \land \psi)^d  : \equiv  \varphi^d \lor \psi^d 	$;
	\item
	$(\varphi \lor \psi)^d  : \equiv  \varphi^d \land \psi^d 	$;
\item
	$(\varphi \to \psi)^d  : \equiv  \varphi \land \psi^d 	$;
	\item $(\forall x \varphi(x))^d : \equiv \exists x \varphi^d(x)$; 
	\item $(\exists x \varphi(x))^d : \equiv \forall x \varphi^d(x)$. 
\end{enumerate}
In fact, $\varphi^d$ is $\HA$-equivalent to $\neg \varphi$ for quantifier-free $\varphi$, and hence,  $\varphi^d$ is $\HA$-equivalent to $\varphi^\bot$ for prenex $\varphi$.
On the one hand, clauses 3 and 4 in Proposition \ref{prop:DUAL_BASIC} hold for the operation $(\cdot)^d$.
On the other hand, for clause 2, $\varphi \to \left({\varphi^{d}} \right)^d$ is not provable in $\HA$ for some (non-prenex) $\varphi$ whereas the converse is always provable in $\HA$.
\end{remark}

In contrast to Proposition \ref{prop:DUAL_BASIC}.(3), the formula $\neg \varphi \to \varphi^\bot$ cannot be proved in $\HA$ even for some prenex $\varphi$.
For example, $\neg \mathrm{Con}(\HA) \to \mathrm{Con}(\HA)^\bot$ is not provable in $\HA$, where $\mathrm{Con}(\HA)$ is a conventional $\Pi_1$ consistency statement of $\HA$ (cf.~\cite[Section 4]{Smor}). 
Thus, we introduce the following principle. 

\begin{defn}[The dual principles]
Let $\Gamma$ be any set of formulas in prenex normal form. 
\begin{tabbing}
\hspace{25mm} \= \hspace{40mm} \= \hspace{50mm} \kill
$\DUAL{\Gamma}$ \> $\neg \varphi \to \varphi^\bot$ \> ($\varphi \in \Gamma$) 
\end{tabbing}
\end{defn}

The principle $\DUAL{\Sigma_1}$ is provable in $\HA$. 

\begin{prop}\label{prop:S1DUAL}
$\HA \vdash \DUAL{\Sigma_1}$. 
\end{prop}
\begin{proof}
Let $\varphi \equiv \exists x \psi$ be any $\Sigma_1$ formula where $\psi$ is $\Sigma_0$. 
Then $\varphi^\bot$ is $\forall x\, \neg \psi$, and hence $\neg \varphi$ is equivalent to $\varphi^\bot$ over $\HA$. 
\end{proof}

\begin{prop}\label{prop:DNE_DUAL}
The following are equivalent over $\HA$: 
\begin{enumerate}
	\item $\DUAL{\Sigma_{k+1}}$. 
	\item $\DUAL{\Pi_k}$. 
	\item $\DNE{\Sigma_k}$. 
\end{enumerate}
\end{prop}
\begin{proof}
It is trivial that $\HA + \DUAL{\Sigma_{k+1}}$ proves $\DUAL{\Pi_k}$ because $\Pi_k \subseteq \Sigma_{k+1}$. 

We prove $\HA + \DUAL{\Pi_k} \vdash \DNE{\Sigma_k}$. 
Let $\varphi$ be any $\Sigma_k$ formula. 
By Proposition \ref{prop:DUAL_BASIC}.(3), we have $\HA \vdash \varphi^\bot \to \neg \varphi$. 
Then $\HA \vdash \neg \neg \varphi \to \neg \varphi^\bot$. 
Since $\varphi^\bot$ is $\Pi_k$ by Proposition \ref{prop:DUAL_BASIC}.(1), $\HA + \DUAL{\Pi_k}$ proves $\neg \varphi^\bot \to \varphi^{\bot \bot}$. 
By Proposition \ref{prop:DUAL_BASIC}.(2), we conclude $\HA + \DUAL{\Pi_k} \vdash \neg \neg \varphi \to \varphi$. 

Finally, we prove $\HA + \DNE{\Sigma_k} \vdash \DUAL{\Sigma_{k+1}}$ by induction on $k$. 
The case $k = 0$ follows from Proposition \ref{prop:S1DUAL}. 
Suppose that the statement holds for all $k' < k+1$, and we prove $\HA + \DNE{\Sigma_{k+1}} \vdash \DUAL{\Sigma_{k+2}}$. 

Let $\exists x \forall y \psi$ be any $\Sigma_{k+2}$ formula where $\psi$ is $\Sigma_k$. 
Since $\HA + \DNE{\Sigma_k}$ proves $\neg \neg \psi \to \psi$, we have $\HA + \DNE{\Sigma_k} \vdash \neg\, \exists x \forall y \psi \to \neg\, \exists x \forall y\, \neg \neg \psi$. 
Then, 
\[
	\HA + \DNE{\Sigma_k} \vdash \neg\, \exists x \forall y \psi \to \forall x\, \neg \neg\, \exists y\, \neg \psi.
\]
By induction hypothesis, $\HA + \DNE{\Sigma_{k-1}} \vdash \neg \psi \to \psi^\bot$. 
Then, 
\[
	\HA + \DNE{\Sigma_k} \vdash \neg\, \exists x \forall y \psi \to \forall x\, \neg \neg\, \exists y \psi^\bot.
\]
Since $\exists y \psi^\bot \equiv (\forall y \psi)^\bot$ is $\Sigma_{k+1}$, 
\[
	\HA + \DNE{\Sigma_{k+1}} \vdash \neg\, \exists x \forall y \psi \to \forall x (\forall y\psi)^\bot.
\]
We conclude $\HA + \DNE{\Sigma_{k+1}} \vdash \neg\, \exists x \forall y \psi \to (\exists x \forall y \psi)^\bot$. 
\end{proof}

From Propositions \ref{prop:DUAL_BASIC}.(3) and \ref{prop:DNE_DUAL}, we obtain Fact \ref{fact:DUAL}.

We may introduce the following $\Delta_k$-variations of the dual principle. 

\begin{defn}[$\Delta_k$ dual principles]\leavevmode
\begin{tabbing}
\hspace{25mm} \= \hspace{40mm} \= \hspace{50mm} \kill
$\DUAL{\Delta_k}^\Sigma$ \> $(\varphi \leftrightarrow \psi) \to (\neg \varphi \to \varphi^\bot)$ \> ($\varphi \in \Sigma_k$ and $\psi \in \Pi_k$) \\
$\DUAL{\Delta_k}^\Pi$ \> $(\varphi \leftrightarrow \psi) \to (\neg \psi \to \psi^\bot)$ \> ($\varphi \in \Sigma_k$ and $\psi \in \Pi_k$) 
\end{tabbing}
\end{defn}

However, each of them is trivially equivalent to the corresponding original dual principle. 

\begin{prop}\label{prop:DDUAL}\leavevmode
\begin{enumerate}
	\item $\DUAL{\Delta_k}^\Sigma$ is equivalent to $\DUAL{\Sigma_k}$ over $\HA$; 
	\item $\DUAL{\Delta_k}^\Pi$ is equivalent to $\DUAL{\Pi_k}$ over $\HA$. 
\end{enumerate}
\end{prop}
\begin{proof}
1. $\HA + \DUAL{\Sigma_k}$ obviously proves $\DUAL{\Delta_k}^\Sigma$. 
On the other hand, let $\varphi$ be any $\Sigma_k$ formula. 
Then $\HA \vdash \neg \varphi \to (\varphi \leftrightarrow \bot)$. 
Hence $\HA + \DUAL{\Delta_k}^\Sigma$ proves $\neg \varphi \to (\neg \varphi \to \varphi^\bot)$. 
We conclude $\HA + \DUAL{\Delta_k}^\Sigma \vdash \neg \varphi \to \varphi^\bot$. 

2 is proved in a similar way. 
\end{proof}

Thus it follows from Proposition \ref{prop:DNE_DUAL} that $\DUAL{\Delta_k}^\Sigma$ and $\DUAL{\Delta_k}^\Pi$ are $\HA$-equivalent to $\DNE{\Sigma_{k-1}}$ and $\DNE{\Sigma_k}$, respectively. 
In fact, $\DUAL{\Delta_1}^\Pi$ corresponds to the principle (VIb) in \cite{FIN}, and it is proved to be $\HA$-equivalent to $\DNE{\Sigma_1}$ (see \cite[Proposition 1]{FIN}). 

\subsection{The weak dual principles}

In this subsection, we investigate weak variations of the dual principle, which we call the weak dual principles. 
\begin{defn}[The weak dual principles]\label{defn:WDUAL}
Let $\Gamma$ be any set of formulas in prenes normal form. 
\begin{tabbing}
\hspace{25mm} \= \hspace{40mm} \= \hspace{50mm} \kill
$\WDUAL{\Gamma}$ \> $\neg \varphi^\bot \to \neg \neg \varphi$ \> ($\varphi \in \Gamma$) 
\end{tabbing}
\end{defn}

Of course $\DUAL{\Gamma}$ implies $\WDUAL{\Gamma}$ over $\HA$. 
It is known that $\DNE{\Sigma_1}$ is not provable in $\HA$ (cf.~\cite{ABHK}), and so is $\DUAL{\Pi_1}$ by Proposition \ref{prop:DNE_DUAL}. 
On the other hand, the following proposition shows that $\WDUAL{\Pi_1}$ is $\HA$-provable. 

\begin{prop}\label{prop:1WDUAL}\leavevmode
\begin{enumerate}
	\item $\HA \vdash \WDUAL{\Sigma_1}$; 
	\item $\HA \vdash \WDUAL{\Pi_1}$. 
\end{enumerate}
\end{prop}
\begin{proof}
1. This follows from Proposition \ref{prop:S1DUAL}. 

2. Let $\forall x \varphi$ be any $\Pi_1$ formula where $\varphi$ is $\Sigma_0$. 
Since $\neg (\forall x \varphi)^\bot \equiv \neg\, \exists x\, \neg \varphi$, we have 
\begin{align*}
	\HA \vdash \neg (\forall x \varphi)^\bot & \to \forall x\, \neg \neg \varphi, \\ 
& \to \forall x \varphi, \tag{because $\varphi \in \Sigma_0$}\\
& \to \neg \neg\, \forall x \varphi. \qedhere
\end{align*}
\end{proof}

Unlike the situation of the dual principles, we show that $\WDUAL{\Sigma_{k+1}}$ and $\WDUAL{\Pi_{k+1}}$ are equivalent over $\HA$. 

\begin{prop}\label{prop:WDUAL}
The following are equivalent over $\HA$: 
\begin{enumerate}
	\item $\WDUAL{\Sigma_{k+1}}$. 
	\item $\WDUAL{\Pi_{k+1}}$. 
	\item $\DNS{\Sigma_k}$. 
\end{enumerate}
\end{prop}
\begin{proof}
First, we prove $\HA + \WDUAL{\Sigma_{k+1}} \vdash \DNS{\Sigma_k}$. 
Let $\varphi$ be any $\Sigma_k$ formula. 
Since $\exists x \varphi^\bot$ is $\Sigma_{k+1}$, 
\[
	\HA + \WDUAL{\Sigma_{k+1}} \vdash \neg (\exists x \varphi^\bot)^\bot \to \neg \neg\, \exists x \varphi^\bot.
\]
By Propositions \ref{prop:DUAL_BASIC}.(2) and \ref{prop:DUAL_BASIC}.(3), $\HA + \WDUAL{\Sigma_{k+1}} \vdash \neg\, \forall x \varphi \to \neg \neg\, \exists x\, \neg \varphi$, and thus $\HA + \WDUAL{\Sigma_{k+1}}$ proves $\neg\, \exists x\, \neg \varphi \to \neg \neg\, \forall x \varphi$. 
Then, we obtain  
\[
	\HA + \WDUAL{\Sigma_{k+1}} \vdash \forall x\, \neg \neg \varphi \to \neg \neg\, \forall x \varphi. 
\]

Secondly, we prove $\HA + \WDUAL{\Pi_{k+1}} \vdash \DNS{\Sigma_k}$. 
Let $\varphi$ be any $\Sigma_k$ formula. 
By Proposition \ref{prop:DUAL_BASIC}.(3), $(\forall x \varphi)^\bot \equiv \exists x \varphi^\bot$ implies $\exists x\, \neg \varphi$ in $\HA$. 
Thus $\HA \vdash \neg\, \exists x\, \neg \varphi \to \neg (\forall x \varphi)^\bot$. 
Since $\forall x \varphi$ is $\Pi_{k+1}$, we obtain
\[
	\HA + \WDUAL{\Pi_{k+1}} \vdash \forall x\, \neg \neg \varphi \to \neg \neg\, \forall x \varphi. 
\]

Finally, we show that $\HA + \DNS{\Sigma_k}$ proves both $\WDUAL{\Sigma_{k+1}}$ and $\WDUAL{\Pi_{k+1}}$ by induction on $k$. 
The case $k=0$ follows from Proposition \ref{prop:1WDUAL}. 
Suppose that the statement holds for $k$, and we prove 
\begin{itemize}
	\item [{\rm (i)}] $\HA + \DNS{\Sigma_{k+1}} \vdash \WDUAL{\Sigma_{k+2}}$; and 
	\item [{\rm (ii)}] $\HA + \DNS{\Sigma_{k+1}} \vdash \WDUAL{\Pi_{k+2}}$. 
\end{itemize}

(i): Let $\exists x \varphi$ be any $\Sigma_{k+2}$ formula where $\varphi$ is $\Pi_{k+1}$. 
By induction hypothesis, 
\[
	\HA + \DNS{\Sigma_k} \vdash \neg \varphi^\bot \to \neg \neg \varphi.  
\]
Then, $\HA + \DNS{\Sigma_k}$ proves the formula $\neg \varphi \to \neg \neg \varphi^\bot$, and hence it proves $\forall x\, \neg \varphi \to \forall x\, \neg \neg \varphi^\bot$. 
Since $\varphi^\bot$ is $\Sigma_{k+1}$, by applying $\DNS{\Sigma_{k+1}}$, we obtain
\[
	\HA + \DNS{\Sigma_{k+1}} \vdash \forall x\, \neg \varphi \to \neg \neg\, \forall x \varphi^\bot.
\]
Then $\HA + \DNS{\Sigma_{k+1}} \vdash \neg\, \forall x \varphi^\bot \to \neg\, \forall x\, \neg \varphi$. 
Therefore we conclude
\[
	\HA + \DNS{\Sigma_{k+1}} \vdash \neg (\exists x \varphi)^\bot \to \neg \neg\, \exists x \varphi.
\]

(ii): Let $\forall x \varphi$ be any $\Pi_{k+2}$ formula where $\varphi$ is $\Sigma_{k+1}$. 
Since $\neg (\forall x \varphi)^\bot \equiv \neg\, \exists x \varphi^\bot$ implies $\forall x\, \neg \varphi^\bot$ in $\HA$, by induction hypothesis, we obtain
\[
	\HA + \DNS{\Sigma_k} \vdash \neg (\forall x \varphi)^\bot \to \forall x\, \neg \neg \varphi. 
\]
Since $\varphi$ is $\Sigma_{k+1}$, we conclude
\[
	\HA + \DNS{\Sigma_{k+1}} \vdash \neg (\forall x \varphi)^\bot \to \neg \neg\, \forall x \varphi. \qedhere
\]
\end{proof}

As in the case of the dual principles, we can introduce the $\Delta_k$-variations of the weak dual principle, namely, $\WDUAL{\Delta_k}^\Sigma$ and $\WDUAL{\Delta_k}^\Pi$. 
Notice that any instance of $\WDUAL{\Gamma}$ is $\HA$-equivalent to a formula of the form $\neg \varphi \to \neg \neg \varphi^\bot$. 
Then, as in the proof of Proposition \ref{prop:DDUAL}, it is shown that $\WDUAL{\Delta_k}^\Sigma$ and $\WDUAL{\Delta_k}^\Pi$ are equivalent to $\WDUAL{\Sigma_k}$ and $\WDUAL{\Pi_k}$ over $\HA$, respectively. 
So they are also equivalent to $\DNS{\Sigma_{k-1}}$ by Proposition \ref{prop:WDUAL}.

\section{The law of excluded middle}\label{section:LEM}

In this section, we investigate variations of the law of excluded middle. 
This section consists of two subsections. 
First, we investigate the law of excluded middle with respect to duals. 
Secondly, we investigate the law of excluded middle for negated formulas.

\subsection{The law of excluded middle with respect to duals}

From the observations in Section \ref{section:dual}, $\varphi^\bot$ is stronger than $\neg \varphi$. 
Hence by replacing $\neg \varphi$ in $\LEM{\Gamma}$ with $\varphi^\bot$, we can expect to get a stronger principle. 
As an example of an application of the investigations in Section \ref{section:dual}, in this subsection, we study this kind of variation of the law of excluded middle.

\begin{defn}[The law of excluded middle with respect to duals]\label{defn:LEMD}
Let $\Gamma$ be any set of formulas in prenex normal form. 
\begin{tabbing}
\hspace{25mm} \= \hspace{40mm} \= \hspace{50mm} \kill
$\LEM{\Gamma}^\bot$ \> $\varphi \lor \varphi^\bot$ \> ($\varphi \in \Gamma$) \\
$\LEM{\Delta_k}^{\bot, \Sigma}$ \> $(\varphi \leftrightarrow \psi) \to \varphi \lor \varphi^\bot$ \> ($\varphi \in \Sigma_k$ and $\psi \in \Pi_k$) \\
$\LEM{\Delta_k}^{\bot, \Pi}$ \> $(\varphi \leftrightarrow \psi) \to \psi \lor \psi^\bot$ \> ($\varphi \in \Sigma_k$ and $\psi \in \Pi_k$) 
\end{tabbing}
\end{defn}

The principle $\LEM{\Delta_1}^{\bot, \Pi}$ corresponds to the principle (IIIb) in \cite{FIN} and to the principle $\LEM{\Delta_b}$ in \cite{FuKo}. 
The following fact is already known. 

\begin{fact}[Fujiwara, Ishihara and Nemoto {\cite[Proposition 1]{FIN}}]\label{fact:FIN1}
The following are equivalent over $\HA$: 
\begin{enumerate}
	\item $\LEM{\Delta_1}^{\bot, \Pi}$. 
	\item $\DNE{\Sigma_1}$. 
\end{enumerate}  
\end{fact}

The following proposition shows interrelations between the laws of excluded middle and their counterparts with respect to duals. 

\begin{prop}\label{prop:LEM}
Let $\Gamma$ be any set of formulas in prenex normal form. 
\begin{enumerate}
	\item $\LEM{\Gamma}^\bot$ is equivalent to $\LEM{\Gamma} + \DUAL{\Gamma}$ over $\HA$; 
	\item $\HA + \LEM{\Delta_k}^{\bot, \Sigma} \vdash \LEM{\Delta_k}$; 
	\item $\HA + \LEM{\Delta_k}^{\bot, \Pi} \vdash \LEM{\Delta_k}$; 
	\item $\HA + \LEM{\Delta_k} + \DUAL{\Sigma_k} \vdash \LEM{\Delta_k}^{\bot, \Sigma}$; 
	\item $\HA + \LEM{\Delta_k} + \DUAL{\Pi_k} \vdash \LEM{\Delta_k}^{\bot, \Pi}$. 
\end{enumerate}
\end{prop}
\begin{proof}
1. By Proposition \ref{prop:DUAL_BASIC}.(3), $\HA + \LEM{\Gamma}^\bot \vdash \LEM{\Gamma}$. 
Also $\HA + \LEM{\Gamma}^\bot \vdash \DUAL{\Gamma}$ is evident because $\HA$ proves $\varphi \lor \varphi^\bot \to (\neg \varphi \to \varphi^\bot)$. 
On the other hand, $\HA + \LEM{\Gamma} + \DUAL{\Gamma} \vdash \LEM{\Gamma}^\bot$ is easily obtained. 

Clauses 2, 3, 4 and 5 are proved similarly.
\end{proof}

From Proposition \ref{prop:LEM}, we obtain the exact strengths of the principles defined in Definition \ref{defn:LEMD}. 

\begin{prop}\label{prop:LEMD}\leavevmode
\begin{enumerate}
	\item $\LEM{\Sigma_k}^\bot$ is equivalent to $\LEM{\Sigma_k}$ over $\HA$; 
	\item $\LEM{\Pi_k}^\bot$ is equivalent to $\LEM{\Sigma_k}$ over $\HA$; 
	\item $\LEM{\Delta_k}^{\bot, \Sigma}$ is equivalent to $\LEM{\Delta_k}$ over $\HA$; 
	\item $\LEM{\Delta_k}^{\bot, \Pi}$ is equivalent to $\DNE{\Sigma_k}$ over $\HA$. 
\end{enumerate}
\end{prop}
\begin{proof}
1. By Proposition \ref{prop:LEM}.(1), $\LEM{\Sigma_k}^\bot$ is equivalent to $\LEM{\Sigma_k} + \DUAL{\Sigma_k}$. 
Since $\HA + \LEM{\Sigma_k}$ proves $\DUAL{\Sigma_k}$ by Fact \ref{fact:ABHK} and Proposition \ref{prop:DNE_DUAL}, $\LEM{\Sigma_k}^\bot$ is equivalent to $\LEM{\Sigma_k}$. 

2. Since $\HA + \LEM{\Pi_k}^\bot$ proves $\varphi^\bot \lor \varphi^{\bot \bot}$ for each $\Sigma_k$ sentence $\varphi$, $\HA + \LEM{\Pi_k}^\bot \vdash \LEM{\Sigma_k}^\bot$ follows from Proposition \ref{prop:DUAL_BASIC}.(2). 
In a similar way, we have $\HA + \LEM{\Sigma_k}^\bot \vdash \LEM{\Pi_k}^\bot$. 
Hence by clause 1, $\LEM{\Pi_k}^\bot$ equivalent to $\LEM{\Sigma_k}$ over $\HA$. 

3. Since $\HA + \LEM{\Delta_k} \vdash \DNE{\Sigma_{k-1}}$, this is immediately obtained from Propositions \ref{prop:DNE_DUAL}, \ref{prop:LEM}.(2) and \ref{prop:LEM}.(4). 

4. Since $\HA + \DNE{\Sigma_k}$ proves $\LEM{\Delta_k}$ and $\DUAL{\Pi_k}$ by Fact \ref{fact:ABHK} and Proposition \ref{prop:DNE_DUAL}, we obtain $\HA + \DNE{\Sigma_k} \vdash \LEM{\Delta_k}^{\bot, \Pi}$ by Proposition \ref{prop:LEM}.(5). 

On the other hand, we prove $\HA + \LEM{\Delta_k}^{\bot, \Pi} \vdash \DNE{\Sigma_k}$. 
Let $\varphi$ be any $\Sigma_k$ formula. 
Since $\neg \neg \varphi \to \neg \varphi^\bot$ is $\HA$-provable by Proposition \ref{prop:DUAL_BASIC}.(3), we obtain $\HA \vdash \neg \neg \varphi \to (\varphi^\bot \leftrightarrow \bot)$. 
Since $\varphi^\bot \in \Pi_k$ and $\bot \in \Sigma_k$, 
\[
	\HA + \LEM{\Delta_k}^{\bot, \Pi} \vdash \neg \neg \varphi \to \varphi^\bot \lor \varphi^{\bot \bot}. 
\]
Since $\HA + \LEM{\Delta_k}^{\bot, \Pi} \vdash \neg \neg \varphi \to \neg \varphi \lor \varphi$ by Proposition \ref{prop:DUAL_BASIC}, we conclude that $\HA + \LEM{\Delta_k}^{\bot, \Pi}$ proves $\neg \neg \varphi \to \varphi$. 
\end{proof}

Proposition \ref{prop:LEMD}.(4) is a generalization of Fact \ref{fact:FIN1}.

\subsection{The law of excluded middle for negated formulas}

In this subsection, we investigate the law of excluded middle for negated formulas, which are investigated in \cite{F20,FIN} for $k=1$. 

\begin{defn}[The law of excluded middle for negated formulas]
Let $\Gamma$ be any set of formulas. 
\begin{tabbing}
\hspace{25mm} \= \hspace{40mm} \= \hspace{50mm} \kill
$\LEM{\n{\Gamma}}$ \> $\neg \varphi \lor \neg \neg \varphi$ \> ($\varphi \in \Gamma$, in other words, $\neg \varphi \in \n{\Gamma}$) \\
$\LEM{\n{\Delta_k}}$ \> $(\varphi \leftrightarrow \psi) \to \neg \varphi \lor \neg \neg \varphi$ \> ($\varphi \in \Sigma_k$ and $\psi \in \Pi_k$)
\end{tabbing}
\end{defn}

Although the definition of $\LEM{\n{\Gamma}}$ is included in Definition \ref{defn:principles}, we defined it individually to pay attention to its properties.
The principle $\LEM{\n{\Delta_1}}$ corresponds to the principle (IVa) in \cite{FIN} and $\Delta_a$-$\mathbf{WLEM}$ in \cite{F20}. 
The following fact is already obtained. 

\begin{fact}[Fujiwara, Ishihara and Nemoto {\cite[Proposition 3]{FIN}}]\label{fact:FIN2}
The following are equivalent over $\HA$:
\begin{enumerate}
	\item $\LEM{\n{\Delta_1}}$. 
	\item $\LEM{\Delta_1}$. 
\end{enumerate}
\end{fact}

Obviously, $\LEM{\n{\Gamma}}$ is weaker than $\LEM{\Gamma}$, and we obtain the following proposition. 
Proposition \ref{prop:NLEM}.(2) is a generalization of Fact \ref{fact:FIN2}. 

\begin{prop}\label{prop:NLEM}
Let $\Gamma$ be any set of formulas. 
\begin{enumerate}
	\item $\LEM{\n{\Gamma}} + \DNE{\Gamma}$ is equivalent to $\LEM{\Gamma}$ over $\HA$; 
	\item $\LEM{\n{\Delta_k}} + \DNE{\Sigma_{k-1}}$ is equivalent to $\LEM{\Delta_k}$ over $\HA$; 
	\item $\HA + \LEM{\n{\Sigma_k}} \vdash \LEM{\n{\Delta_k}}$; 
	\item $\HA + \LEM{\n{\Pi_k}} \vdash \LEM{\n{\Delta_k}}$. 
\end{enumerate}
\end{prop}
\begin{proof}
1. This follows from Fact \ref{fact:LEM_DNE}. 

2. This is a consequence of Facts \ref{fact:LEM_DNE} and \ref{fact:ABHK}. 

3 and 4 are obvious. 
\end{proof}

From Fact \ref{fact:ABHK}, $\LEM{\Sigma_k}$ and $\LEM{\Pi_k}$ are equivalent modulo $\DNE{\Sigma_k}$. 
We prove an analogous result concerning $\LEM{\n{\Sigma_k}}$ and $\LEM{\n{\Pi_k}}$. 

\begin{prop}\label{prop:WLEMSP}
The following are equivalent over $\HA + \DNS{\Sigma_{k-1}}$:
\begin{enumerate}
	\item $\LEM{\n{\Sigma_k}}$. 
	\item $\LEM{\n{\Pi_k}}$. 
\end{enumerate}
\end{prop}
\begin{proof}
First, we show $\HA + \DNS{\Sigma_{k-1}} + \LEM{\n{\Sigma_k}} \vdash \LEM{\n{\Pi_k}}$. 
Let $\varphi$ be any $\Pi_k$ formula.  
Since $\varphi^\bot$ is $\Sigma_k$, we have 
\[
	\HA + \LEM{\n{\Sigma_k}} \vdash \neg \varphi^\bot \lor \neg \neg \varphi^\bot.
\]
Then, $\HA + \LEM{\n{\Sigma_k}} \vdash \neg \varphi^\bot \lor \neg \varphi$ by Proposition \ref{prop:DUAL_BASIC}.(3). 
Since $\WDUAL{\Pi_k}$ is equivalent to $\DNS{\Sigma_{k-1}}$ over $\HA$ by Proposition \ref{prop:WDUAL}, we obtain
\[
	\HA + \DNS{\Sigma_{k-1}} + \LEM{\n{\Sigma_k}} \vdash \neg \neg \varphi \lor \neg \varphi. 
\]

In a similar way, it is proved that $\HA + \DNS{\Sigma_{k-1}} + \LEM{\n{\Pi_k}}$ proves $\LEM{\n{\Sigma_k}}$ because $\WDUAL{\Sigma_k}$ is also equivalent to $\DNS{\Sigma_{k-1}}$ over $\HA$ by Proposition \ref{prop:WDUAL}. 
\end{proof}

From Fact \ref{fact:ABHK}.(6), Propositions \ref{prop:DNS1}.(1), \ref{prop:NLEM} and \ref{prop:WLEMSP}, we obtain the following corollaries. 

\begin{cor}\label{cor:NLEM1}
The following are equivalent over $\HA$:
\begin{enumerate}
	\item $\LEM{\Pi_k}$. 
	\item $\LEM{\n{\Sigma_k}} + \DNE{\Sigma_{k-1}}$. 
	\item $\LEM{\n{\Pi_k}} + \DNE{\Sigma_{k-1}}$. 
\end{enumerate}
\end{cor}

\begin{cor}\label{cor:NLEM2}
The following are equivalent over $\HA$:
\begin{enumerate}
	\item $\LEM{\Sigma_k}$. 
	\item $\LEM{\n{\Sigma_k}} + \DNE{\Sigma_k}$. 
	\item $\LEM{\n{\Pi_k}} + \DNE{\Sigma_k}$. 
\end{enumerate}
\end{cor}

\section{De Morgan's law}\label{section:DML}

In this section, we extensively investigate principles based on de Morgan's law. 

\begin{defn}[De Morgan's law]
Let $\Gamma$ and $\Theta$ be any sets of formulas. 
\begin{tabbing}
\hspace{25mm} \= \hspace{60mm} \= \hspace{50mm} \kill
$\DML{(\Gamma, \Theta)}$ \> $\neg(\varphi \land \psi) \to \neg \varphi \lor \neg \psi$ \> ($\varphi \in \Gamma$ and $\psi \in \Theta$) \\
$\DML{\Delta_k}$ \> $(\varphi \leftrightarrow \varphi') \land (\psi \leftrightarrow \psi')$ \>  \\
 \> \hspace{0.5in} $\to (\neg (\varphi \land \psi) \to \neg \varphi \lor \neg \psi)$ \> ($\varphi, \psi \in \Sigma_k$ and $\varphi', \psi' \in \Pi_k$) \\
$\DML{(\Delta_k, \Theta)}$ \> $(\varphi \leftrightarrow \varphi') \to (\neg(\varphi \land \psi) \to \neg \varphi \lor \neg \psi)$ \> ($\varphi \in \Sigma_k$, $\varphi' \in \Pi_k$ and $\psi \in \Theta$) 
\end{tabbing}
\end{defn}

Several variations of $\DML{\Delta_1}$ are extensively investigated in \cite{F20}. 
As in the case of the law of excluded middle, we also deal with the principles of the forms $\DML{(\n{\Gamma}, \Theta)}$, $\DML{(\n{\Delta_k}, \Theta)}$, and so on. 
Of course, $\DML{(\Gamma, \Theta)}$ and $\DML{(\Theta, \Gamma)}$ are equivalent. 

This section consists of four subsections. 
First, we investigate several basic implications between the principles. 
Secondly, we study the interrelationship between de Morgan's law and the contrapositive version of the collection principle. 
Thirdly, $\Delta_k$ and $\n{\Delta_k}$ variants of de Morgan's law are explored. 
Finally, we investigate de Morgan's law with respect to duals.

\subsection{Basic implications}

In this subsection, we organize several versions of de Morgan's law. 
Some arguments in this subsection for $k=1$ can be found in \cite{F20}. 
The following proposition is trivially obtained. 

\begin{prop}\label{prop:DML0}\leavevmode
Let $\Gamma \in \{\Sigma_k, \Pi_k\}$ and $\Theta$ be any set of formulas. 
\begin{enumerate}
	\item $\HA + \DML{(\Gamma, \Theta)} \vdash \DML{(\Delta_k, \Theta)}$;  
	\item $\HA + \DML{(\n{\Gamma}, \Theta)} \vdash \DML{(\n{\Delta_k}, \Theta)}$. 
\end{enumerate}
\end{prop}

We show that $\LEM{\n{\Gamma}}$ and $\LEM{\n{\Delta_k}}$ are stronger than several versions of de Morgan's law. 

\begin{prop}\label{prop:DML1}
Let $\Gamma$ and $\Theta$ be any sets of formulas. 
\begin{enumerate}
	\item $\HA + \LEM{\n{\Gamma}} \vdash \DML{(\Gamma, \Theta)}$; 
	\item $\HA + \LEM{\n{\Gamma}} \vdash \DML{(\n{\Gamma}, \Theta)}$; 
	\item $\HA + \LEM{\n{\Delta_k}} \vdash \DML{(\Delta_k, \Theta)}$; 
	\item $\HA + \LEM{\n{\Delta_k}} \vdash \DML{(\n{\Delta_k}, \Theta)}$. 
\end{enumerate}
\end{prop}
\begin{proof}
1. Let $\varphi \in \Gamma$ and $\psi \in \Theta$. 
Since $\HA \vdash \neg (\varphi \land \psi) \to \neg (\neg \neg \varphi \land \psi)$, we get
\[
	\HA \vdash (\neg \varphi \lor \neg \neg \varphi) \to (\neg (\varphi \land \psi) \to \neg \varphi \lor \neg \psi).
\]
It follows that $\HA + \LEM{\n{\Gamma}}$ proves $\DML{(\Gamma, \Theta)}$. 

2, 3 and 4 are proved as for clause 1. 
\end{proof}

\begin{cor}\label{cor:nLEM_DML}
\leavevmode
\begin{enumerate}
	\item For any set $\Gamma$ of formulas, $\HA + \LEM{\n{\Gamma}}$ proves $\DML{\Gamma}$ and $\DML{\n{\Gamma}}$; 
	\item $\HA + \LEM{\n{\Delta_k}}$ proves $\DML{\Delta_k}$ and $\DML{\n{\Delta_k}}$. 
\end{enumerate}
\end{cor}

Conversely, we show that the principles $\LEM{\n{\Gamma}}$ and $\LEM{\n{\Delta_k}}$ are equivalent to some variations of de Morgan's law. 

\begin{prop}\label{prop:WLEM_DML}
For any set $\Gamma$ of formulas, the following are equivalent over $\HA$:
\begin{enumerate}
	\item $\LEM{\n{\Gamma}}$. 
	\item $\DML{(\Gamma, \n{\Gamma})}$. 
\end{enumerate}
\end{prop}
\begin{proof}
By Proposition \ref{prop:DML1}, $\HA + \LEM{\n{\Gamma}} \vdash \DML{(\Gamma, \n{\Gamma})}$. 
On the other hand, let $\varphi$ be any $\Gamma$ formula. 
Since $\HA \vdash \neg (\varphi \land \neg \varphi)$, we obtain $\HA + \DML{(\Gamma, \n{\Gamma})} \vdash \neg \varphi \lor \neg \neg \varphi$. 
\end{proof}

\begin{prop}\label{prop:WLEM_DML2}
For $\Gamma \in \{\Sigma_k, \Pi_k\}$, the following are equivalent over $\HA$:
\begin{enumerate}
	\item $\LEM{\n{\Delta_k}}$. 
	\item $\DML{(\Delta_k, \n{\Gamma})}$. 
	\item $\DML{(\n{\Delta_k}, \Gamma)}$. 
	\item $\DML{(\Delta_k, \n{\Delta_k})}$. 
\end{enumerate}
\end{prop}
\begin{proof}
By Proposition \ref{prop:DML1}, $\LEM{\n{\Delta_k}}$ entails $\DML{(\Delta_k, \n{\Gamma})}$ and $\DML{(\n{\Delta_k}, \Gamma)}$. 
By Proposition \ref{prop:DML0}, each of $\DML{(\Delta_k, \n{\Gamma})}$ and $\DML{(\n{\Delta_k}, \Gamma)}$ implies $\DML{(\Delta_k, \n{\Delta_k})}$. 
On the other hand, we can show that $\HA + \DML{(\Delta_k, \n{\Delta_k})}$ proves $\LEM{\n{\Delta_k}}$ as in the proof of Proposition \ref{prop:WLEM_DML}. 
\end{proof}

Here we investigate several equivalences of some variations of de Morgan's law over the theory $\HA + \DNS{\Sigma_{k-1}}$. 

\begin{prop}\label{prop:DML_DNS}
Let $\Theta$ be any set of formulas. 
\begin{enumerate}
	\item $\DML{(\n{\Sigma_k}, \Theta)}$ is equivalent to $\DML{(\Pi_k, \Theta)}$ over $\HA + \DNS{\Sigma_{k-1}}$; 
	\item $\DML{(\n{\Pi_k}, \Theta)}$ is equivalent to $\DML{(\Sigma_k, \Theta)}$ over $\HA + \DNS{\Sigma_{k-1}}$. 
\end{enumerate}
\end{prop}
\begin{proof}
Recall that each of $\WDUAL{\Sigma_k}$ and $\WDUAL{\Pi_k}$ is $\HA$-equivalent to $\DNS{\Sigma_{k-1}}$ (Proposition \ref{prop:WDUAL}). 
Then for any $\varphi \in \Sigma_k$ and $\psi \in \Pi_k$, $\HA + \DNS{\Sigma_{k-1}}$ proves $\neg \varphi^\bot \leftrightarrow\, \neg \neg \varphi$ and $\neg \psi^\bot \leftrightarrow\, \neg \neg \psi$. 
Then clauses 1 and 2 follow from this observation and the fact that $\HA$ proves $\neg (\xi \land \delta) \leftrightarrow\, \neg(\neg \neg \xi \land \delta)$. 
\end{proof}

From Proposition \ref{prop:DML_DNS}, we obtain several equivalences over $\HA + \DNS{\Sigma_{k-1}}$. 

\begin{cor}\label{cor:DML2}\leavevmode
\begin{enumerate}
	\item $\LEM{\n{\Sigma_k}}$, $\LEM{\n{\Pi_k}}$, $\DML{(\Sigma_k, \n{\Sigma_k})}$, $\DML{(\Pi_k, \n{\Pi_k})}$, $\DML{(\Sigma_k, \Pi_k)}$ and $\DML{(\n{\Sigma_k}, \n{\Pi_k})}$ are equivalent over $\HA + \DNS{\Sigma_{k-1}}$; 
	\item $\DML{\Sigma_k}$, $\DML{(\Sigma_k, \n{\Pi_k})}$ and $\DML{\n{\Pi_k}}$ are equivalent over $\HA + \DNS{\Sigma_{k-1}}$; 
	\item $\DML{\Pi_k}$, $\DML{(\Pi_k, \n{\Sigma_k})}$ and $\DML{\n{\Sigma_k}}$ are equivalent over $\HA + \DNS{\Sigma_{k-1}}$; 
	\item For $\Gamma \in \{\Sigma_k, \Pi_k, \n{\Sigma_k}, \n{\Pi_k}\}$, each of $\DML{(\Delta_k, \Gamma)}$ and $\DML{(\n{\Delta_k}, \Gamma)}$ is equivalent to $\LEM{\n{\Delta_k}}$ over $\HA + \DNS{\Sigma_{k-1}}$. 
\end{enumerate}
\end{cor}
\begin{proof}
1. This is a consequence of Propositions \ref{prop:WLEMSP}, \ref{prop:WLEM_DML} and \ref{prop:DML_DNS}. 

2 and 3 are immediate from Proposition \ref{prop:DML_DNS}. 

4. The principles $\DML{(\Delta_k, \Sigma_k)}$, $\DML{(\Delta_k, \Pi_k)}$, $\DML{(\n{\Delta_k}, \n{\Sigma_k})}$ and $\DML{(\n{\Delta_k}, \n{\Pi_k})}$ are equivalent to $\DML{(\Delta_k, \n{\Pi_k})}$, $\DML{(\Delta_k, \n{\Sigma_k})}$, $\DML{(\n{\Delta_k}, \Pi_k)}$ and $\DML{(\n{\Delta_k}, \Sigma_k)}$ over $\HA + \DNS{\Sigma_{k-1}}$, respectively. 
Then, by Proposition \ref{prop:WLEM_DML2}, each of them is equivalent to $\LEM{\n{\Delta_k}}$. 
\end{proof}

From Corollaries \ref{cor:NLEM1}, \ref{cor:NLEM2}, \ref{cor:DML2} and Proposition \ref{prop:WLEM_DML}, we also obtain the following. 

\begin{cor}\label{cor:DML_LEM}
Let $P$ be one of $\DML{(\Sigma_k, \n{\Sigma_k})}$, $\DML{(\Pi_k, \n{\Pi_k})}$, $\DML{(\Sigma_k, \Pi_k)}$ and $\DML{(\n{\Sigma_k}, \n{\Pi_k})}$. 
\begin{enumerate}
	\item $P + \DNE{\Sigma_{k-1}}$ is equivalent to $\LEM{\Pi_k}$ over $\HA$; 
	\item $P + \DNE{\Sigma_k}$ is equivalent to $\LEM{\Sigma_k}$ over $\HA$. 
\end{enumerate}
\end{cor}

The following corollary follows from Propositions \ref{prop:WLEM_DML2}, \ref{prop:NLEM}.(2) and Corollary \ref{cor:DML2}.(4). 

\begin{cor}\label{cor:DML_LEM3}
Let $\Gamma \in \{\Sigma_k, \Pi_k, \n{\Sigma_k}, \n{\Pi_k}\}$. 
Let $P$ be one of the principles $\DML{(\Delta_k, \Gamma)}$, $\DML{(\n{\Delta_k}, \Gamma)}$ and $\DML{(\Delta_k, \n{\Delta_k})}$. 
Then $P + \DNE{\Sigma_{k-1}}$ is equivalent to $\LEM{\Delta_k}$ over $\HA$. 
\end{cor}

We get the following corollary. 

\begin{cor}\label{cor:PDML_DLEM}
Let $\Gamma \in \{\Sigma_k, \Pi_k, \n{\Sigma_k}, \n{\Pi_k}\}$. 
\begin{enumerate}
	\item $\HA + \DML{\Gamma} + \DNS{\Sigma_{k-1}} \vdash \LEM{\n{\Delta_k}}$;  
	\item $\HA + \DML{\Gamma} + \DNE{\Sigma_{k-1}} \vdash \LEM{\Delta_k}$. 
\end{enumerate}
\end{cor}
\begin{proof}
1. Since $\DML{\Gamma}$ implies $\DML{(\Delta_k, \Gamma)}$ by Proposition \ref{prop:DML0}, the statement immediately follows from Corollary \ref{cor:DML2}.(4). 

2. This follows from Corollary \ref{cor:DML_LEM3}. 
\end{proof}

Corollary \ref{cor:PDML_DLEM}.(2) generalizes Fact \ref{fact:FIN3}. 
Also we generalize Fact \ref{fact:Ishi}.(1).

\begin{prop}\label{prop:SDNE_DML}
$\HA + \DNE{\Sigma_k} \vdash \DML{\Pi_k}$. 
\end{prop}
\begin{proof}
Since $\HA + \DNE{\Sigma_k}$ proves $\DNS{\Sigma_{k-1}}$, it is sufficient to show that $\HA + \DNE{\Sigma_k} \vdash \DML{\n{\Sigma_k}}$ by Corollary \ref{cor:DML2}.(3). 
Let $\varphi$ and $\psi$ be any $\Sigma_k$ formulas. 
Since $\HA \vdash \neg (\neg \varphi \land \neg \psi) \to \neg \neg (\varphi \lor \psi)$ and $\varphi \lor \psi$ is $\HA$-equivalent to some $\Sigma_k$ formula, we obtain
\[
	\HA + \DNE{\Sigma_k} \vdash \neg (\neg \varphi \land \neg \psi) \to \varphi \lor \psi. 
\]
Therefore
\[
	\HA + \DNE{\Sigma_k} \vdash \neg (\neg \varphi \land \neg \psi) \to \neg \neg \varphi \lor \neg \neg \psi. \qedhere
\]
\end{proof}

By combining Corollary \ref{cor:PDML_DLEM}.(2) and Proposition \ref{prop:SDNE_DML}, we obtain a proof of Fact \ref{fact:ABHK}.(4).

\subsection{The collection principles and de Morgan's law}

In this subsection, we investigate the so-called collection principles. 
The following proposition is stated in \cite{Burr}. 

\begin{prop}\label{prop:Coll1}
For any formula $\varphi(y, z)$, 
\[
	\HA \vdash \forall y < x \, \exists z \, \varphi(y, z) \to \exists w \, \forall y < x \, \exists z < w \, \varphi(y, z).
\] 
\end{prop}
\begin{proof}
Let $\psi(x)$ be the formula
\[
	\forall y < x \, \exists z \, \varphi(y, z) \to \exists w \, \forall y < x \, \exists z < w \, \varphi(y, z), 
\]
and this proposition is proved by applying the induction axiom for $\psi(x)$. 
%
%
%
\end{proof}

We introduce the following contrapositive version of the collection principle. 

\begin{defn}[The contrapositive collection principles]
Let $\Gamma$ be any set of formulas. 
\begin{tabbing}
\hspace{20mm} \= \hspace{75mm} \= \hspace{50mm} \kill
$\COLL{\Gamma}$ \> $\forall w \, \exists y < x \, \forall z < w \, \varphi(y, z) \to \exists y < x \, \forall z \, \varphi(y, z)$ \> ($\varphi(y, z) \in \Gamma$) 
\end{tabbing}
\end{defn}

\begin{prop}\label{prop:Coll2}
The following are equivalent over $\HA$: 
\begin{enumerate}
	\item $\COLL{\Pi_{k+1}}$. 
	\item $\COLL{\Sigma_k}$. 
\end{enumerate}
\end{prop}
\begin{proof}
By using a primitive recursive pairing function, it is easy to show that for any $\Sigma_k$ formula $\varphi(y, z_0, z_1)$, $\HA + \COLL{\Sigma_k}$ proves
\begin{equation}\label{eqColl2}
	\forall w \, \exists y < x \, \forall z_0 < w \, \forall z_1 < w\, \varphi(y, z_0, z_1) \to \exists y < x \, \forall z_0\, \forall z_1 \, \varphi(y, z_0, z_1). 
\end{equation}
From this observation, the equivalence of $\COLL{\Sigma_k}$ and $\COLL{\Pi_{k+1}}$ immediately follows. 
\end{proof}

The following proposition extends \cite[Corollary 4.5]{FuKa}. 

\begin{prop}\label{prop:Coll3}
$\HA + \DML{\Sigma_{k+1}} + \DNE{\Sigma_k} \vdash \COLL{\Pi_{k+1}}$. 
\end{prop}
\begin{proof}
We simultaneously prove the following two statements by induction on $k$: 
\begin{itemize}
	\item [{\rm (i)}] $\HA + \DML{\Sigma_{k+1}} + \DNE{\Sigma_k} \vdash \COLL{\Pi_{k+1}}$; 
	\item [{\rm (ii)}] For any $\Pi_{k+1}$ formula $\varphi(y)$, there exists a $\Pi_{k+1}$ formula $\psi(x)$ such that 
\[
	\HA + \DML{\Sigma_{k+1}} + \DNE{\Sigma_k} \vdash \exists y < x\, \varphi(y) \leftrightarrow \psi(x). 
\]
\end{itemize}
 
We suppose that our statements hold for all $k' < k$, and we prove (i) and (ii). 

(i): Prior to proving our statement, we show that for any $\Pi_k$ formula $\varphi(y, z)$, 
\begin{equation}\label{eqColl3}
	\HA + \DML{\Sigma_{k+1}} + \DNE{\Sigma_{k-1}} \vdash \neg\, \forall y < x\, \exists z\, \varphi(y, z) \to \exists y < x\, \forall z\, \neg \varphi(y, z), 
\end{equation}
which is a generalization of \cite[Lemma 4.4]{FuKa}. 

Let $\psi(x)$ be the formula
\[
	\neg\, \forall y < x\, \exists z\, \varphi(y, z) \to \exists y < x\, \forall z\, \neg \varphi(y, z), 
\]
and we show that $\forall x \psi(x)$ is derivable by applying the induction axiom for $\psi(x)$. 
Since $\HA \vdash \neg y < 0$, we have $\HA \vdash \forall y < 0\, \exists z\, \varphi(y, z)$. 
Thus we obviously obtain $\HA \vdash \psi(0)$. 

We prove induction step. 
We have
\[
	\HA \vdash \neg\, \forall y \leq x\, \exists z\, \varphi(y, z) \to \neg (\forall y < x\, \exists z\, \varphi(y, z) \land \exists z\, \varphi(x, z)). 
\]
By Proposition \ref{prop:Coll1}, the formula $\forall y < x\, \exists z\, \varphi(y, z)$ is $\HA$-equivalent to the formula $\exists w\, \forall y < x\, \exists z < w\, \varphi(y, z)$. 
If $k = 0$, the formula $\exists z < w\, \varphi(y, z)$ is $\HA$-provably equivalent to some $\Pi_0$ formula $\rho(y, w)$. 
If $k > 0$, by induction hypothesis (ii) for $k-1$, the formula ${\exists z < w\, \varphi(y, z)}$ is equivalent to some $\Pi_k$ formula $\rho(y, w)$ in $\HA + \DML{\Sigma_k} + \DNE{\Sigma_{k-1}}$. 
Also ${\exists w\, \forall y < x\, \rho(y, w)}$ is $\HA$-equivalent to a $\Sigma_{k+1}$ formula. 
Thus ${\forall y < x\, \exists z\, \varphi(y, z)}$ can be regarded as a $\Sigma_{k+1}$  formula in $\HA + \DML{\Sigma_k} + \DNE{\Sigma_{k-1}}$. 
Then $\HA + \DML{\Sigma_{k+1}} + \DNE{\Sigma_{k-1}}$ proves
\[
	\neg\, \forall y \leq x\, \exists z\, \varphi(y, z) \to \neg\, \forall y < x\, \exists z\, \varphi(y, z) \lor \neg\, \exists z\, \varphi(x, z).
\]
Hence it also proves
\[
	\psi(x) \land \neg\, \forall y \leq x\, \exists z\, \varphi(y, z) \to \exists y < x\, \forall z\, \neg \varphi(y, z) \lor \forall z\, \neg \varphi(x, z).
\]
It follows that the theory proves
\[
	\psi(x) \land \neg\, \forall y \leq x\, \exists z\, \varphi(y, z) \to \exists y \leq x\, \forall z\, \neg \varphi(y, z).
\]
This means $\HA + \DML{\Sigma_{k+1}} + \DNE{\Sigma_{k-1}} \vdash \psi(x) \to \psi(x+1)$. 
We have proved (\ref{eqColl3}). 

We prove $\HA + \DML{\Sigma_{k+1}} + \DNE{\Sigma_k} \vdash \COLL{\Pi_{k+1}}$. 
It suffices to prove $\COLL{\Sigma_k}$ by Proposition \ref{prop:Coll2}. 
Let $\varphi(y, z)$ be any $\Sigma_k$ formula. 
By Proposition \ref{prop:Coll1} for the formula $\varphi^\bot(y, z)$, we have
\[
	\HA \vdash \neg\, \exists w \, \forall y < x \, \exists z < w \, \varphi^\bot(y, z) \to \neg\, \forall y < x \, \exists z \, \varphi^\bot(y, z). 
\]
In the light of Proposition \ref{prop:DUAL_BASIC}.(3), we obtain
\[
	\HA \vdash \forall w \, \exists y < x \, \forall z < w \, \varphi(y, z) \to \neg\, \exists w \, \forall y < x \, \exists z < w \, \varphi^\bot(y, z).
\]
Therefore
\[
	\HA \vdash \forall w \, \exists y < x \, \forall z < w \, \varphi(y, z) \to \neg\, \forall y < x \, \exists z \, \varphi^\bot(y, z). 
\]
Since $(\varphi(y, z))^\bot$ is $\Pi_k$, from (\ref{eqColl3}), we obtain that $\HA + \DML{\Sigma_{k+1}} + \DNE{\Sigma_{k-1}}$ proves
\[
	\forall w \, \exists y < x \, \forall z < w \, \varphi(y, z) \to \exists y < x \, \forall z \, \neg \varphi^\bot(y, z). 
\]
Since $\DNE{\Sigma_k}$ proves $\DUAL{\Pi_k}$, 
we conclude that $\HA + \DML{\Sigma_{k+1}} + \DNE{\Sigma_k}$ proves
\[
	\forall w \, \exists y < x \, \forall z < w \, \varphi(y, z) \to \exists y < x \, \forall z \, \varphi(y, z)
\]
by Proposition \ref{prop:DUAL_BASIC}.(2).
This completes the proof of (i). 

(ii): Let $\forall z \varphi(y, z)$ be any $\Pi_{k+1}$ formula where $\varphi(y, z)$ is $\Sigma_k$. 
Since $\varphi^\bot(y, z)$ is $\Pi_k$, by induction hypothesis (ii) for $k-1$, there exists a $\Pi_k$ formula $\psi(y, w)$ such that
\[
	\HA + \DML{\Sigma_k} + \DNE{\Sigma_{k-1}} \vdash \exists z < w\, \varphi^\bot(y, z) \leftrightarrow \psi(y, w).
\]
This is also the case for $k = 0$. 
Then
\[
	\HA + \DML{\Sigma_k} + \DNE{\Sigma_{k-1}} \vdash \forall z < w\, \neg \varphi^\bot(y, z) \leftrightarrow\, \neg \psi(y, w).
\]
Since $\DNE{\Sigma_k}$ implies $\DUAL{\Pi_k}$, we obtain
\[
	\HA + \DML{\Sigma_k} + \DNE{\Sigma_k} \vdash \forall z < w\, \varphi(y, z) \leftrightarrow \psi^\bot(y, w).
\]
By (i), we have that $\HA + \DML{\Sigma_{k+1}} + \DNE{\Sigma_k}$ proves
\[
	\exists y < x\, \forall z\, \varphi(y, z) \leftrightarrow \forall w\, \exists y < x\, \forall z < w\, \varphi(y, z). 
\]
Therefore we obtain that $\HA + \DML{\Sigma_{k+1}} + \DNE{\Sigma_k}$ also proves
\[
	\exists y < x\, \forall z\, \varphi(y, z) \leftrightarrow \forall w\, \exists y < x\, \psi^\bot(y, w). 
\]
This completes the proof of (ii). 
\end{proof}

\begin{remark}
By Proposition \ref{prop:Coll2}, $\COLL{\Pi_0}$ is equivalent to $\COLL{\Pi_1}$ over $\HA$. 
We will show in Proposition \ref{prop:Coll4} that $\HA + \COLL{\Pi_1} \vdash \DML{\Sigma_1}$. 
Therefore $\HA \nvdash \COLL{\Pi_0}$ because it is known that $\HA \nvdash \DML{\Sigma_1}$ (cf.~\cite{ABHK}). 
Thus the statement of Proposition \ref{prop:Coll3} for $k = -1$ does not holds. 
\end{remark}

\begin{cor}\label{cor:BQ}\leavevmode
\begin{enumerate}
	\item For any $\Pi_k$ formula $\varphi(y)$, there exists a $\Pi_k$ formula $\psi(x)$ such that
\[
	\HA + \DML{\Sigma_k} + \DNE{\Sigma_{k-1}} \vdash \exists y < x\, \varphi(y) \leftrightarrow \psi(x); 
\]
	\item For any $\Sigma_k$ formula $\varphi(y)$, there exists a $\Sigma_k$ formula $\psi(x)$ such that
\[
	\HA + \DML{\Sigma_{k-1}} + \DNE{\Sigma_{k-2}} \vdash \forall y < x\, \varphi(y) \leftrightarrow \psi(x). 
\]
\end{enumerate}
\end{cor}
\begin{proof}
1. For $k = 0$, this is trivial. 
For $k > 0$, the statement is already proved in the proof of Proposition \ref{prop:Coll3}. 

2. Since the statement obviously holds for $k=0$, we may assume $k > 0$. 
Let $\exists z \varphi(y, z)$ be any $\Sigma_k$ formula where $\varphi(y, z)$ is $\Pi_{k-1}$. 
By Proposition \ref{prop:Coll1}, we have 
\[
	\HA \vdash \forall y < x\, \exists z\, \varphi(y, z) \leftrightarrow \exists w\, \forall y < x\, \exists z < w\, \varphi(y, z). 
\]
By clause 1, there exists a $\Pi_{k-1}$ formula $\psi(y, w)$ such that
\[
	\HA + \DML{\Sigma_{k-1}} + \DNE{\Sigma_{k-2}} \vdash \exists z < w\, \varphi(y, z) \leftrightarrow \psi(y, w).
\]
Hence 
\[
	\HA + \DML{\Sigma_{k-1}} + \DNE{\Sigma_{k-2}} \vdash \forall y < x\, \exists z\, \varphi(y, z) \leftrightarrow \exists w\, \forall y < x\, \psi(y, w).
\]
Since $\exists w\, \forall y < x\, \psi(y, w)$ is obviously equivalent to a $\Sigma_k$ formula, this completes our proof of clause 2.  
\end{proof}

Corollary \ref{cor:BQ} is very useful for exploring principles containing bounded quantifiers. 
For instance, it can be applied to the study of the least number principle. 

\begin{defn}[The least number principle]
Let $\Gamma$ be a set of formulas. 
\begin{tabbing}
\hspace{15mm} \= \hspace{75mm} \= \hspace{50mm} \kill
$\LN{\Gamma}$ \> $\exists x \varphi(x) \to \exists x (\varphi(x) \land \forall y < x\, \neg \varphi(y))$ \> ($\varphi \in \Gamma$) 
\end{tabbing}
\end{defn}

\begin{thm}\label{thm:LN}
Let $\Gamma$ be either $\Sigma_k$ or $\Pi_k$.  
Then $\LN{\Gamma}$ and $\LEM{\Gamma}$ are equivalent over $\HA$. 
\end{thm}
\begin{proof}
First, we prove $\HA + \LN{\Gamma} \vdash \LEM{\Gamma}$. 
Let $\varphi$ be any $\Gamma$ formula and let $\psi(x)$ be a $\Gamma$ formula $\HA$-equivalent to $\varphi \lor \, 0 < x$, where $x$ does not occur freely in $\varphi$. 
Notice that $0 < x\, \land\, \forall y < x\, \neg \psi(y)$ implies $\neg \psi(0)$ which implies $\neg \varphi$. 
Hence we have
\[
	\HA \vdash (\varphi \lor\, 0 < x) \land \forall y < x\, \neg \psi(y) \to \varphi \lor \neg \varphi,
\]
and thus
\[
	\HA \vdash \exists x (\psi(x) \land \forall y < x\, \neg \psi(y)) \to \varphi \lor \neg \varphi. 
\]
Since $\HA \vdash \exists x \psi(x)$, we have $\HA + \LN{\Gamma} \vdash \exists x (\psi(x) \land \forall y < x\, \neg \psi(y))$. 
Therefore we obtain $\HA + \LN{\Gamma} \vdash \varphi \lor \neg \varphi$. 

Secondly, we prove $\HA + \LEM{\Pi_k} \vdash \LN{\Pi_k}$. 
A proof for $\HA + \LEM{\Sigma_k} \vdash \LN{\Sigma_k}$ is similar. 
Let $\varphi(x)$ be any $\Pi_k$ formula, and let $\psi(z)$ be the formula
\[
	\exists x < z\, \varphi(x) \to \exists x < z\, (\varphi(x) \land \forall y < x\, \neg \varphi(y)).
\]
We prove $\HA + \LEM{\Pi_k} \vdash \forall z \psi(z)$ by applying the induction axiom for $\psi(z)$. 
Since $\HA \vdash \neg\, \exists x < 0\, \varphi(x)$, we obtain $\HA \vdash \psi(0)$. 

We prove induction step. 
Notice $\HA + \LEM{\Pi_k}$ proves $\DML{\Sigma_k} + \DNE{\Sigma_{k-1}}$ by Corollary \ref{cor:NLEM1} and Proposition \ref{prop:DML1}.(1). 
Thus by Corollary \ref{cor:BQ}.(1), the formula $\exists x < z\, \varphi(x)$ is equivalent to some $\Pi_k$ formula in $\HA + \LEM{\Pi_k}$. 
Therefore
\begin{equation}\label{eqLN}
	\HA + \LEM{\Pi_k} \vdash \exists x < z\, \varphi(x) \lor \neg\, \exists x < z\, \varphi(x).
\end{equation}
Since $\HA \vdash \exists x \leq z\, \varphi(x) \leftrightarrow (\exists x < z\, \varphi(x) \lor \varphi(z))$, we obtain
\[
	\HA \vdash \exists x \leq z\, \varphi(x) \land \neg\, \exists x < z\, \varphi(x) \to \varphi(z) \land \forall x < z\, \neg \varphi(x), 
\]
and hence
\[
	\HA \vdash \exists x \leq z\, \varphi(x) \land \neg\, \exists x < z\, \varphi(x) \to \exists x \leq z\, (\varphi(x) \land \forall y < x\, \neg \varphi(y)).  
\]
On the other hand, we obviously obtain
\[
	\HA \vdash \psi(z) \land \exists x < z\, \varphi(x) \to \exists x \leq z\, (\varphi(x) \land \forall y < x\, \neg \varphi(y)). 
\]
Then by (\ref{eqLN}), we have
\[
	\HA + \LEM{\Pi_k} \vdash \psi(z) \land \exists x \leq z\, \varphi(x) \to \exists x \leq z\, (\varphi(x) \land \forall y < x\, \neg \varphi(y)).  
\]
It follows $\HA + \LEM{\Pi_k} \vdash \psi(z) \to \psi(z+1)$. 
We have completed our proof. 
\end{proof}

By using Corollary \ref{cor:BQ} and Theorem \ref{thm:LN}, we are able to generalize Fact \ref{fact:Ishi}.(2). 
The proof is similar to that of the implication $2 \Rightarrow 1$ of \cite[Proposition 2]{FIN}.

\begin{prop}\label{prop:PSDML}
$\HA + \DML{\Sigma_k} + \DNE{\Sigma_{k-1}} \vdash \DML{\Pi_k}$. 
\end{prop}
\begin{proof}
We may assume $k > 0$. 
Let $\forall x \varphi(x)$ and $\forall y \psi(y)$ be any $\Pi_k$ formulas where $\varphi(x)$ and $\psi(y)$ are $\Sigma_{k-1}$. 
We define the formulas $\xi(x)$ and $\eta(y)$ as follows: 
\begin{itemize}
	\item $\xi(x) : \equiv \forall z < x (\varphi(z) \land \psi(z)) \land \varphi^\bot(x)$; 
	\item $\eta(y) : \equiv \forall z < y (\varphi(z) \land \psi(z)) \land \psi^\bot(y) \land \varphi(y)$. 
\end{itemize}
Since $\varphi(z) \land \psi(z)$ is $\HA$-equivalent to a $\Sigma_{k-1}$ formula, by Corollary \ref{cor:BQ}.(2), $\forall z < x (\varphi(z) \land \psi(z))$ is equivalent to some $\Sigma_{k-1}$ formula in $\HA + \DML{\Sigma_{k-2}} + \DNE{\Sigma_{k-3}}$. 
Thus the formula $\exists x \xi(x)$ is equivalent to a $\Sigma_k$ formula in the theory. 
Similarly, $\exists y \eta(y)$ is also equivalent to some $\Sigma_k$ formula in the theory. 

By the definitions of $\xi(x)$ and $\eta(y)$, we obtain
\begin{itemize}
	\item $\HA \vdash \xi(x) \land \eta(y) \land x \leq y \to \varphi^\bot(x) \land \varphi(x)$, and 
	\item $\HA \vdash \xi(x) \land \eta(y) \land y < x \to \psi(y) \land \psi^\bot(y)$. 
\end{itemize}
Thus by Proposition \ref{prop:DUAL_BASIC}.(4) and $\HA \vdash x \leq y \lor y < x$, we have that $\HA$ proves $\neg (\exists x \xi(x) \land \exists y \eta(y))$. 
Then from the above observations, we obtain
\begin{equation}\label{eqDML1}
	\HA + \DML{\Sigma_k} + \DNE{\Sigma_{k-3}} \vdash \neg\, \exists x \xi(x) \lor \neg\, \exists y \eta(y). 
\end{equation}

Note that $\HA + \DML{\Sigma_k} + \DNE{\Sigma_{k-1}}$ proves $\DUAL{\Sigma_{k-1}}$, $\DUAL{\Pi_{k-1}}$ and $\LEM{\Pi_{k-1}}$. 
Then $\HA + \DML{\Sigma_k} + \DNE{\Sigma_{k-1}}$ proves 
\begin{align*}
	 \exists x\, \neg \varphi(x) & \to \exists x \varphi^\bot(x), \tag{by $\DUAL{\Sigma_{k-1}}$}\\
	& \to \exists x[\varphi^\bot(x) \land \forall z < x\, \neg \varphi^\bot(z)], \tag{by $\LEM{\Pi_{k-1}}$ and Theorem \ref{thm:LN}}\\
	& \to \exists x[\varphi^\bot(x) \land \forall z < x\, \varphi(z)]. \tag{by $\DUAL{\Pi_{k-1}}$ and Proposition \ref{prop:DUAL_BASIC}.(2)}
\end{align*}
Hence, by the definition of the formula $\xi(x)$, we have
\[
	\HA + \DML{\Sigma_k} + \DNE{\Sigma_{k-1}} \vdash \exists x\, \neg \varphi(x) \land \forall y \psi(y) \to \exists x \xi(x). 
\]
Since $\HA \vdash \neg\, \exists x\, \neg \varphi(x) \to \forall x\, \neg \neg \varphi(x)$ and $\HA + \DNE{\Sigma_{k-1}}$ implies $\DNS{\Sigma_{k-1}}$ by Proposition \ref{prop:DNS1}.(1), we obtain
\[
	\HA + \DML{\Sigma_k} + \DNE{\Sigma_{k-1}} \vdash \forall y \psi(y) \land \neg\, \exists x \xi(x) \to \neg \neg\, \forall x \varphi(x). 
\]
On the other hand, 
\[
	\HA \vdash \neg(\forall x \varphi(x) \land \forall y \psi(y)) \land \forall y \psi(y) \to \neg\, \forall x \varphi(x). 
\]
Therefore we obtain
\begin{equation}\label{eqDML2}
	\HA + \DML{\Sigma_k} + \DNE{\Sigma_{k-1}} \vdash \neg(\forall x \varphi(x) \land \forall y \psi(y)) \land \neg\, \exists x \xi(x) \to \neg\, \forall y \psi(y).  
\end{equation}
In a similar way, we obtain
\begin{equation}\label{eqDML3}
	\HA + \DML{\Sigma_k} + \DNE{\Sigma_{k-1}} \vdash \neg(\forall x \varphi(x) \land \forall y \psi(y)) \land \neg\, \exists y \eta(y) \to \neg\, \forall x \varphi(x).  
\end{equation}
By combining (\ref{eqDML1}), (\ref{eqDML2}) and (\ref{eqDML3}), we conclude
\[
	\HA + \DML{\Sigma_k} + \DNE{\Sigma_{k-1}} \vdash \neg(\forall x \varphi(x) \land \forall y \psi(y)) \to \neg\, \forall x \varphi(x) \lor \neg\, \forall y \psi(y). \qedhere
\]
\end{proof}

Finally, we prove that the converse of Proposition \ref{prop:Coll3} also holds. 
This is closely related to \cite[Theorem 4.5]{Burr}. 

\begin{prop}\label{prop:Coll4}
$\HA + \COLL{\Pi_k} \vdash \DML{\Sigma_k} + \LEM{\Sigma_{k-1}}$. 
\end{prop}
\begin{proof}
We prove by induction on $k$. 
For $k=0$, our statement obviously holds. 
Suppose that the statement holds for $k$, and we prove the case of $k+1$. 
We prove the following two statements: 
\begin{itemize}
	\item [{\rm (i)}] $\HA + \COLL{\Pi_{k+1}} \vdash \LEM{\Sigma_k}$; 
	\item [{\rm (ii)}] $\HA + \COLL{\Pi_{k+1}} \vdash \DML{\Sigma_{k+1}}$.  
\end{itemize}

(i): Let $\exists x \varphi$ be any $\Sigma_k$ formula where $\varphi$ is $\Pi_{k-1}$. 
By induction hypothesis, $\HA + \COLL{\Pi_k} \vdash \DML{\Sigma_k} + \LEM{\Sigma_{k-1}}$. 
By Fact \ref{fact:ABHK}, $\HA + \COLL{\Pi_k}$ also proves $\LEM{\Pi_{k-1}}$ and $\DNE{\Sigma_{k-1}}$. 
It follows from Corollary \ref{cor:BQ}.(1), we have that $\exists x < z\, \varphi$ is equivalent to some $\Pi_{k-1}$ formula in $\HA + \COLL{\Pi_k}$. 
Then by applying $\LEM{\Pi_{k-1}}$, we obtain 
\[
	\HA + \COLL{\Pi_k} \vdash \exists x < z\, \varphi \lor \neg\, \exists x < z\, \varphi. 
\]
Then
\[
	\HA + \COLL{\Pi_k} \vdash \exists w < 2\, [(w = 0 \to \exists x < z\, \varphi) \land (w = 1 \to \neg\, \exists x < z\, \varphi)]. 
\]
Since $\HA + \COLL{\Pi_k}$ proves $\DUAL{\Pi_{k-1}}$, we obtain
\[
	\HA + \COLL{\Pi_k} \vdash \exists w <2\, [(w = 0 \to \exists x < z\, \varphi) \land (w= 1\to \forall x < z\, \varphi^\bot)]. 
\]
Hence
\[
	\HA + \COLL{\Pi_k} \vdash \forall z\, \exists w <2\, \forall x < z\, [(w = 0 \to \exists x \varphi) \land (w= 1\to \varphi^\bot)]. 
\]
Since $(w = 0 \to \exists x \varphi) \land (w= 1\to \varphi^\bot)$ is equivalent to some $\Sigma_k$ formula, by Proposition \ref{prop:Coll2}, 
\[
	\HA + \COLL{\Pi_{k+1}} \vdash \exists w <2\, \forall x\, [(w = 0 \to \exists x \varphi) \land (w= 1\to \varphi^\bot)]. 
\]
Then 
\[
	\HA + \COLL{\Pi_{k+1}} \vdash \exists w <2\, [(w = 0 \to \exists x \varphi) \land (w= 1\to \forall x \varphi^\bot)]. 
\]
Thus we obtain $\HA + \COLL{\Pi_{k+1}} \vdash \exists x \varphi \lor \neg\, \exists x \varphi$ by Proposition \ref{prop:DUAL_BASIC}.(3). 
This means $\HA + \COLL{\Pi_{k+1}} \vdash \LEM{\Sigma_k}$. 

(ii): Let $\exists x \varphi$ and $\exists y \psi$ be any $\Sigma_{k+1}$ formulas where $\varphi$ and $\psi$ are $\Pi_k$. 
We have $\HA \vdash \neg(\exists x \varphi \land \exists y \psi) \to \neg (\exists x < z\, \varphi \land \exists y < z\, \psi)$. 
From (i), we have that $\HA + \COLL{\Pi_{k+1}}$ proves $\LEM{\Sigma_k}$. 
By Fact \ref{fact:ABHK}, Propositions \ref{prop:PSDML} and \ref{prop:DNE_DUAL}, the theory also proves $\DML{\Sigma_k}$, $\DNE{\Sigma_k}$, $\DML{\Pi_k}$ and $\DUAL{\Pi_k}$. 
Then by Corollary \ref{cor:BQ}.(1), both $\exists x < z\, \varphi$ and $\exists y < z\, \psi$ are equivalent to some $\Pi_k$ formulas in $\HA + \COLL{\Pi_{k+1}}$. 
By applying $\DML{\Pi_k}$, $\HA + \COLL{\Pi_{k+1}}$ proves
\begin{align*}
	\neg(\exists x \varphi \land \exists y \psi) & \to \neg\, \exists x < z\, \varphi \lor \neg\, \exists y < z\, \psi, \\
& \to \exists w < 2\, [(w=0 \to \neg\, \exists x < z\, \varphi) \land (w = 1 \to \neg\, \exists y < z\, \psi)], \\
& \to \exists w < 2\, [(w=0 \to \forall x < z\, \varphi^\bot) \land (w = 1 \to \forall y < z\, \psi^\bot)], \tag{by $\DUAL{\Pi_k}$}\\
& \to \exists w < 2\, \forall x < z\, \forall y < z\, [(w=0 \to \varphi^\bot) \land (w = 1 \to \psi^\bot)]. 
\end{align*}
Thus we have that $\HA + \COLL{\Pi_{k+1}}$ proves
\[
	\neg(\exists x \varphi \land \exists y \psi) \to \forall z\, \exists w < 2\, \forall x < z\, \forall y < z\, [(w=0 \to \varphi^\bot) \land (w = 1 \to \psi^\bot)]. 
\]
Then, in the light of (\ref{eqColl2}), $\HA + \COLL{\Pi_{k+1}}$ proves
\begin{align*}
	\neg(\exists x \varphi \land \exists y \psi) & \to \exists w < 2\, \forall x\, \forall y\, [(w=0 \to \varphi^\bot) \land (w = 1 \to \psi^\bot)], \\
	& \to \exists w < 2\, [(w=0 \to \forall x \varphi^\bot) \land (w = 1 \to \forall y \psi^\bot)], \\
	& \to \exists w < 2\, [(w=0 \to \neg\, \exists x \varphi) \land (w = 1 \to \neg\, \exists y \psi)], \tag{by Proposition \ref{prop:DUAL_BASIC}.(3)}\\
	& \to \neg\, \exists x \varphi \lor \neg\, \exists y \psi.
\end{align*}
Therefore $\HA + \COLL{\Pi_{k+1}} \vdash \DML{\Sigma_{k+1}}$. 
\end{proof}

From Propositions \ref{prop:Coll3}, \ref{prop:Coll4} and Fact \ref{fact:ABHK}, we get the following corollary. 

\begin{cor}\label{cor:PCOLL}
The following are equivalent over $\HA$: 
\begin{enumerate}
	\item $\COLL{\Pi_{k+1}}$. 
	\item $\DML{\Sigma_{k+1}} + \LEM{\Sigma_k}$. 
	\item $\DML{\Sigma_{k+1}} + \DNE{\Sigma_k}$. 
\end{enumerate}
\end{cor}

\subsection{The principles $\DML{\Delta_k}$ and $\DML{\n{\Delta_k}}$}

In this subsection, we mainly investigate the principles $\DML{\Delta_k}$ and $\DML{\n{\Delta_k}}$. 

\begin{prop}\label{prop:DDML1}\leavevmode
\begin{enumerate}
	\item $\HA + \DML{\Delta_{k+1}} + \DNS{\Sigma_{k-1}} \vdash \LEM{\n{\Sigma_k}}$; 
	\item $\HA + \DML{\n{\Delta_{k+1}}} + \DNS{\Sigma_{k-1}} \vdash \LEM{\n{\Sigma_k}}$. 
\end{enumerate} 
\end{prop}
\begin{proof}
Let $\varphi$ be any $\Sigma_k$ formula. 

1. By Proposition \ref{prop:DUAL_BASIC}.(4), $\HA \vdash \neg (\varphi \land \varphi^\bot)$. 
Since both $\varphi$ and $\varphi^\bot$ are $\Delta_{k+1}$, $\HA + \DML{\Delta_{k+1}} \vdash \neg \varphi \lor \neg \varphi^\bot$. 
Then $\HA + \DML{\Delta_{k+1}} + \DNS{\Sigma_{k-1}}$ proves $\neg \varphi \lor \neg \neg \varphi$ by Proposition \ref{prop:WDUAL}. 

2. Since $\HA \vdash \neg(\neg \varphi \land \neg \neg \varphi)$, $\HA + \DNS{\Sigma_{k-1}} \vdash \neg (\neg \varphi \land \neg \varphi^\bot)$. 
Then $\HA + \DML{\n{\Delta_{k+1}}} + \DNS{\Sigma_{k-1}} \vdash \neg \neg \varphi \lor \neg \neg \varphi^\bot$. 
We conclude that the theory proves $\neg \varphi \lor \neg \neg \varphi$. 
\end{proof}

From Corollaries \ref{cor:NLEM1}, \ref{cor:NLEM2} and Proposition \ref{prop:DDML1}, we obtain the following.

\begin{cor}\label{cor:DDML2}
Let $\Gamma \in \{\Delta_{k+1}, \n{\Delta_{k+1}}\}$. 
\begin{enumerate}
	\item $\HA + \DML{\Gamma} + \DNE{\Sigma_{k-1}} \vdash \LEM{\Pi_k}$; 
	\item $\HA + \DML{\Gamma} + \DNE{\Sigma_k} \vdash \LEM{\Sigma_k}$. 
\end{enumerate} 
\end{cor}

Furthermore, we prove the following proposition by adapting the proofs of Proposition \ref{prop:PSDML} and \cite[Lemma 2.14]{F20}. 

\begin{prop}\label{prop:DDML3}
$\HA + \DML{\Delta_k} + \DNE{\Sigma_{k-1}} \vdash \DML{\n{\Delta_k}}$. 
\end{prop}
\begin{proof}
We may assume $k > 0$. 
Let $\exists x \varphi(x)$ and $\exists y \psi(y)$ be any $\Sigma_k$ formulas where $\varphi(x)$ and $\psi(y)$ are $\Pi_{k-1}$, and let $\varphi'$ and $\psi'$ be any $\Pi_k$ formulas. 
Let $\chi$ denote the formula $(\exists x \varphi(x) \leftrightarrow \varphi') \land (\exists y \psi(y) \leftrightarrow \psi')$. 
We define the formulas $\xi(x)$ and $\eta(y)$ as follows: 
\begin{itemize}
	\item $\xi(x) : \equiv \forall z < x (\varphi^\bot(z) \land \psi^\bot(z)) \land \varphi(x)$; 
	\item $\eta(y) : \equiv \forall z < y (\varphi^\bot(z) \land \psi^\bot(z)) \land \psi(y) \land \varphi^\bot(y)$. 
\end{itemize}
As in the proof of Proposition \ref{prop:PSDML}, the formulas $\exists x \xi(x)$ and $\exists y \eta(y)$ are equivalent to some $\Sigma_k$ formulas in the theory $\HA + \DML{\Sigma_{k-2}} + \DNE{\Sigma_{k-3}}$ which is included in $\HA + \DNE{\Sigma_{k-1}}$ by Fact \ref{fact:ABHK}, Corollary \ref{cor:NLEM2} and Proposition \ref{prop:DML1}. 
Also 
\begin{equation}\label{eqDDML1}
	\HA \vdash \neg (\exists x \xi(x) \land \exists y \eta(y)). 
\end{equation}

By Corollary \ref{cor:DDML2}.(1), $\HA + \DML{\Delta_k} + \DNE{\Sigma_{k-2}}$ proves $\LEM{\Pi_{k-1}}$. 
Since $\DNE{\Sigma_{k-1}}$ implies $\DUAL{\Pi_{k-1}}$, by Theorem \ref{thm:LN}, we obtain 
\begin{equation}\label{eqDDML2}
	\HA + \DML{\Delta_k} + \DNE{\Sigma_{k-1}} \vdash \exists x \varphi(x) \to \exists x[\varphi(x) \land \forall z < x\, \varphi^\bot(z)]. 
\end{equation}
In a similar way, we have
\[
	\HA + \DML{\Delta_k} + \DNE{\Sigma_{k-1}} \vdash \exists y < x\, \psi(y) \to \exists y < x\, [\psi(y) \land \forall z < y\, \psi^\bot(z)].
\]
Then by the definition of $\eta(y)$, 
\[
	\HA + \DML{\Delta_k} + \DNE{\Sigma_{k-1}} \vdash \forall z < x\, \varphi^\bot(z) \land \exists z < x\, \psi(z) \to \exists y \eta(y).
\]
From this with (\ref{eqDDML2}), $\HA + \DML{\Delta_k} + \DNE{\Sigma_{k-1}}$ proves
\[
	\exists x \varphi(x) \land \neg\, \exists y \eta(y) \to \exists x[\varphi(x) \land \forall z < x\, \varphi^\bot(z) \land \forall z < x\, \psi^\bot(z)]. 
\]
It follows that the theory proves $\exists x \varphi(x) \land \neg\, \exists y \eta(y) \to \exists x \xi(x)$. 
On the other hand, $\HA$ proves $\exists x \xi(x) \to \exists x \varphi(x) \land \neg\, \exists y \eta(y)$ from (\ref{eqDDML1}). 
Therefore $\HA + \DML{\Delta_k} + \DNE{\Sigma_{k-1}}$ proves
\[
	\chi \to [\exists x \xi(x) \leftrightarrow (\varphi' \land \forall y \eta^\bot(y))]. 
\]
Also $\varphi' \land \forall y \eta^\bot(y)$ is $\HA$-provably equivalent to some $\Pi_k$ formula. 

In a similar way, we obtain that $\HA + \DML{\Delta_k} + \DNE{\Sigma_{k-1}}$ proves
\[
	\chi \to [\exists y \eta(y) \leftrightarrow (\psi' \land \forall x \xi^\bot(x))]
\]
and $\psi' \land \forall x \xi^\bot(x)$ is $\HA$-provably equivalent to some $\Pi_k$ formula. 

Then by applying $\DML{\Delta_k}$ to (\ref{eqDDML1}), 
\begin{equation}\label{eqDDML3}
	\HA + \DML{\Delta_k} + \DNE{\Sigma_{k-1}} \vdash \chi \to \neg\, \exists x \xi(x) \lor \neg\, \exists y \eta(y). 
\end{equation}

From (\ref{eqDDML2}) and the definition of $\xi(x)$, 
\[
	\HA + \DML{\Delta_k} + \DNE{\Sigma_{k-1}} \vdash \exists x \varphi(x) \land \forall y \psi^\bot(y) \to \exists x \xi(x). 
\]
Then 
\[
	\HA + \DML{\Delta_k} + \DNE{\Sigma_{k-1}} \vdash \neg\, \exists x \xi(x) \land \neg\, \exists y \psi(y) \to \neg\, \exists x \varphi(x). 
\]
Therefore we obtain
\begin{equation}\label{eqDDML4}
	\HA + \DML{\Delta_k} + \DNE{\Sigma_{k-1}} \vdash \neg(\neg\, \exists x \varphi(x) \land \neg\, \exists y \psi(y)) \land \neg\, \exists x \xi(x) \to \neg \neg\, \exists y \psi(y).  
\end{equation}
In a similar way, we obtain
\begin{equation}\label{eqDDML5}
	\HA + \DML{\Delta_k} + \DNE{\Sigma_{k-1}} \vdash \neg(\neg\, \exists x \varphi(x) \land \neg\, \exists y \psi(y)) \land \neg\, \exists y \eta(y) \to \neg \neg\, \exists x \varphi(x).  
\end{equation}
By combining (\ref{eqDDML3}), (\ref{eqDDML4}) and (\ref{eqDDML5}), we conclude that $\HA + \DML{\Delta_k} + \DNE{\Sigma_{k-1}}$ proves
\[
	\chi \to [\neg(\neg\, \exists x \varphi(x) \land \neg\, \exists y \psi(y)) \to \neg \neg\, \exists x \varphi(x) \lor \neg \neg\, \exists y \psi(y)]. \qedhere
\]
\end{proof}

\subsection{De Morgan's law with respect to duals}

In \cite{ABHK}, principles based on de Morgan's law with respect to duals are introduced.  

\begin{defn}[De Morgan's law with respect to duals]\label{defn:DDML}
Let $\Gamma$ and $\Theta$ be any sets of formulas in prenex normal form. 
\begin{tabbing}
\hspace{30mm} \= \hspace{60mm} \= \hspace{50mm} \kill
$\DML{\Gamma}^\bot$ \> $\neg(\varphi \land \psi) \to \varphi^\bot \lor \psi^\bot$ \> ($\varphi, \psi \in \Gamma$) \\
$\DML{(\Gamma, \Theta)}^\bot$ \> $\neg(\varphi \land \psi) \to \varphi^\bot \lor \psi^\bot$ \> ($\varphi \in \Gamma$ and $\psi \in \Theta$) \\
$\DML{\Delta_k}^\bot$ \> $(\varphi \leftrightarrow \varphi') \land (\psi \leftrightarrow \psi')$ \>  \\
 \> \hspace{0.5in} $\to (\neg (\varphi \land \psi) \to \varphi^\bot \lor \psi^\bot)$ \> ($\varphi, \psi \in \Sigma_k$ and $\varphi', \psi' \in \Pi_k$) \\
$\DML{(\Delta_k, \Gamma)}^{\bot, \Sigma}$ \> $(\varphi \leftrightarrow \varphi') \to (\neg(\varphi \land \psi) \to \varphi^\bot \lor \psi^\bot)$ \> ($\varphi \in \Sigma_k$, $\varphi' \in \Pi_k$ and $\psi \in \Gamma$) \\
$\DML{(\Delta_k, \Gamma)}^{\bot, \Pi}$ \> $(\varphi \leftrightarrow \varphi') \to (\neg(\varphi \land \psi) \to (\varphi')^\bot \lor \psi^\bot)$ \> ($\varphi \in \Sigma_k$, $\varphi' \in \Pi_k$ and $\psi \in \Gamma$) \\
\end{tabbing}
\end{defn}

Our $\DML{\Sigma_k}^\bot$ is called $\LLPO{\Sigma_k}$ in \cite{ABHK}. 
As in the case of $\LEM{\Gamma}^\bot$ (Proposition \ref{prop:LEM}), we show that the principles defined in Definition \ref{defn:DDML} are exactly de Morgan's laws equipped with the dual principles. 

\begin{prop}\label{prop:LLPO1}
Let $\Gamma$ and $\Theta$ be any sets of formulas in prenex normal form. 
\begin{enumerate}
	\item 	$\DML{(\Gamma, \Theta)}^\bot$ is equivalent to $\DML{(\Gamma, \Theta)} + \DUAL{\Gamma} + \DUAL{\Theta}$ over $\HA$; 
	\item $\DML{(\Delta_k, \Theta)}^{\bot, \Sigma}$ is equivalent to $\DML{(\Delta_k, \Theta)} + \DUAL{\Sigma_k} + \DUAL{\Theta}$ over $\HA$; 
	\item $\DML{(\Delta_k, \Theta)}^{\bot, \Pi}$ is equivalent to $\DML{(\Delta_k, \Theta)} + \DUAL{\Pi_k} + \DUAL{\Theta}$ over $\HA$. 
\end{enumerate}
\end{prop}
\begin{proof}
1. By Proposition \ref{prop:DUAL_BASIC}.(3), $\HA + \DML{(\Gamma, \Theta)}^\bot \vdash \DML{(\Gamma, \Theta)}$. 
Let $\varphi \in \Gamma$. 
Since $\HA \vdash \neg \varphi \to \neg (\varphi \land \neg \bot)$, we have that $\HA  + \DML{(\Gamma, \Theta)}^\bot$ proves the formula $\neg \varphi \to \varphi^\bot \lor (\neg \bot)^\bot$. 
Thus $\HA + \DML{(\Gamma, \Theta)}^\bot \vdash \neg \varphi \to \varphi^\bot$, and this means that $\DUAL{\Gamma}$ is provable. 
Similarly, $\DUAL{\Theta}$ is also provable. 
On the other hand, $\DML{(\Gamma, \Theta)}^\bot$ is easily proved in $\HA + \DML{(\Gamma, \Theta)} + \DUAL{\Gamma} + \DUAL{\Theta}$. 

2 and 3 are proved in a similar way. 
\end{proof}

Summarizing the results so far, we obtain the following corollary.

\begin{cor}\label{cor:LLPO2}\leavevmode
\begin{enumerate}
	\item $\DML{\Sigma_k}^\bot$ is equivalent to $\DML{\Sigma_k} + \DNE{\Sigma_{k-1}}$ over $\HA$; 
	\item $\DML{(\Sigma_k, \Pi_k)}^\bot$ is equivalent to $\LEM{\Sigma_k}$ over $\HA$; 
	\item $\DML{(\Delta_k, \Sigma_k)}^{\bot, \Sigma}$ is equivalent to $\LEM{\Delta_k}$ over $\HA$; 
	\item $\DML{\Delta_k}^\bot$ is equivalent to $\DML{\Delta_k} + \DNE{\Sigma_{k-1}}$ over $\HA$; 
	\item Each of the principles $\DML{\Pi_k}^\bot$, $\DML{(\Delta_k, \Sigma_k)}^{\bot, \Pi}$, $\DML{(\Delta_k, \Pi_k)}^{\bot, \Sigma}$ and $\DML{(\Delta_k, \Pi_k)}^{\bot, \Pi}$ is equivalent to $\DNE{\Sigma_k}$ over $\HA$. 
\end{enumerate}
\end{cor}
\begin{proof}
1. This is a consequence of Propositions \ref{prop:DNE_DUAL} and \ref{prop:LLPO1}.(1). 

2. By Propositions \ref{prop:DNE_DUAL} and \ref{prop:LLPO1}.(1), $\DML{(\Sigma_k, \Pi_k)}^\bot$ is $\HA$-equivalent to $\DML{(\Sigma_k, \Pi_k)} + \DNE{\Sigma_k}$. 
Then it is also $\HA$-equivalent to $\LEM{\Sigma_k}$ by Corollary \ref{cor:DML_LEM}.(2). 

3. From Propositions \ref{prop:DNE_DUAL} and \ref{prop:LLPO1}.(2), $\DML{(\Delta_k, \Sigma_k)}^{\bot, \Sigma}$ is $\HA$-equivalent to $\DML{(\Delta_k, \Sigma_k)} + \DNE{\Sigma_{k-1}}$. 
Then it is $\HA$-equivalent to $\LEM{\Delta_k}$ by Corollary \ref{cor:DML_LEM3}.

4 is proved as in the proof of Proposition \ref{prop:LLPO1}. 

5. By Propositions \ref{prop:DNE_DUAL} and \ref{prop:LLPO1}.(1), $\DML{\Pi_k}^\bot$ is $\HA$-equivalent to $\DML{\Pi_k} + \DNE{\Sigma_k}$. 
Since $\HA + \DNE{\Sigma_k} \vdash \DML{\Pi_k}$ by Proposition \ref{prop:SDNE_DML}, $\DML{\Pi_k}^\bot$ is $\HA$-equivalent to $\DNE{\Sigma_k}$. 
Similarly, each of $\DML{(\Delta_k, \Sigma_k)}^{\bot, \Pi}$, $\DML{(\Delta_k, \Pi_k)}^{\bot, \Sigma}$ and $\DML{(\Delta_k, \Pi_k)}^{\bot, \Pi}$ is $\HA$-equivalent to $\DNE{\Sigma_k}$ because each of them implies $\DNE{\Sigma_k}$ over $\HA$ by Proposition \ref{prop:LLPO1}, and $\HA + \DNE{\Sigma_k}$ proves $\DML{(\Delta_k, \Sigma_k)}$ and $\DML{(\Delta_k, \Pi_k)}$ by Fact \ref{fact:ABHK}.(4) and Proposition \ref{prop:DML1}.(3). 
\end{proof}

In \cite[Theorem 14]{BS14}, it is proved that $\DML{\Sigma_k}^\bot$ is equivalent to $\DML{\Sigma_k} + \LEM{\Sigma_{k-1}}$ over $\HA$. 
This result follows from Corollaries \ref{cor:PCOLL} and \ref{cor:LLPO2}.(1).

\section{The double negation elimination}\label{section:DNE}

In this section, we explore variations of the double negation elimination. 
As in the previous sections, we deal with the principles of forms $\DNE{(\n{\Gamma} \lor \Theta)}$, $\DNE{(\Delta_k \lor \Theta)}$, and so on. 
As in the case of de Morgan's law, $\DNE{(\Gamma \lor \Theta)}$ is obviously equivalent to $\DNE{(\Theta \lor \Gamma)}$. 
Interestingly, de Morgan's law can be seen as a variation of the double negation elimination. 

\begin{prop}\label{prop:DML_DNE}
For any sets $\Gamma$ and $\Theta$ of formulas, the following are equivalent over $\HA$: 
\begin{enumerate}
	\item $\DML{(\Gamma, \Theta)}$. 
	\item $\DNE{(\n{\Gamma} \lor \n{\Theta})}$. 
\end{enumerate}
The analogous equivalences also hold for the versions of $\Delta_k$ and $\n{\Delta_k}$.  
\end{prop}
\begin{proof}
Let $\varphi \in \Gamma$ and $\psi \in \Theta$. 
Since $\HA \vdash \neg(\varphi \land \psi) \leftrightarrow\, \neg \neg (\neg \varphi \lor \neg \psi)$, $\HA$ proves
\[
	[\neg(\varphi \land \psi) \to \neg \varphi \lor \neg \psi] \leftrightarrow [\neg \neg (\neg \varphi \lor \neg \psi) \to \neg \varphi \lor \neg \psi]. 
\]
The last statement is also proved in a similar way. 
\end{proof}

We prove the following basic proposition concerning principles based on the double negation elimination. 

\begin{prop}\label{prop:DNE1}
Let $\Gamma \in \{\Sigma_k, \Pi_k, \Delta_k\}$ and let $\Theta$ be any set of formulas. 
\begin{enumerate}
	\item $\HA + \DNE{(\Gamma \lor \Theta)} \vdash \DNE{\Gamma}$; 
	\item Suppose that for any $\varphi \in \Theta$, there exists $\psi \in \Sigma_k$ such that $\HA + \DNE{\Sigma_k} \vdash \varphi \leftrightarrow \psi$. 
	Then $\DNE{(\Sigma_k \lor \Theta)}$ is equivalent to $\DNE{\Sigma_k}$ over $\HA$; 
	\item $\DNE{(\n{\Sigma_k} \lor \Theta)} + \DNE{\Sigma_{k-1}}$ is equivalent to $\DNE{(\Pi_k \lor \Theta)}$ over $\HA$; 
	\item $\DNE{(\n{\Sigma_k} \lor \Gamma)}$ is equivalent to $\DNE{(\Pi_k \lor \Gamma)}$ over $\HA$; 
	\item $\DNE{(\n{\Pi_k} \lor \Theta)} + \DNE{\Sigma_k}$ is equivalent to $\DNE{(\Sigma_k \lor \Theta)}$ over $\HA$; 
	\item $\DNE{(\dn{\Sigma_k} \lor \Theta)}$ is equivalent to $\DNE{(\n{\Pi_k} \lor \Theta)}$ over $\HA + \DNS{\Sigma_{k-1}}$; 
	\item $\DNE{(\dn{\Sigma_k} \lor \Gamma)}$ is equivalent to $\DNE{(\n{\Pi_k} \lor \Gamma)}$ over $\HA$; 
	\item $\DNE{(\dn{\Pi_k} \lor \Theta)}$ is equivalent to $\DNE{(\n{\Sigma_k} \lor \Theta)}$ over $\HA + \DNS{\Sigma_{k-1}}$; 
	\item $\DNE{(\dn{\Pi_k} \lor \Gamma)}$ is equivalent to $\DNE{(\Pi_k \lor \Gamma)}$ over $\HA$; 
	\item $\DNE{(\dn{\Delta_k} \lor \Theta)} + \DNE{\Sigma_{k-1}}$ is equivalent to $\DNE{(\Delta_k \lor \Theta)}$ over $\HA$; 
	\item $\DNE{(\dn{\Delta_k} \lor \Gamma)}$ is equivalent to $\DNE{(\Delta_k \lor \Gamma)}$ over $\HA$. 
\end{enumerate}
Also these statements hold even if $\Theta \in \{\n{\Delta_k}, \dn{\Delta_k}\}$. 
\end{prop}
\begin{proof}
1. This is because $0 = 0 \in \Theta$ and for any $\varphi \in \Gamma$, $\varphi \lor 0=0$ is $\HA$-provably equivalent to $\varphi$. 

2. From clause 1, $\HA + \DNE{(\Sigma_k \lor \Theta)}$ proves $\DNE{\Sigma_k}$. 
On the other hand, let $\varphi$ and $\psi$ be any $\Sigma_k$ formulas. 
Notice that $\varphi \lor \psi$ is $\HA$-equivalent to $\exists x ((x = 0 \to \varphi) \land (x = 1 \to \psi))$. 
Then it is shown that $\varphi \lor \psi$ is provably equivalent to some $\Sigma_k$ formula in $\HA$ (cf.~\cite[Lemma 4.4]{FuKu}). 
Therefore $\HA + \DNE{\Sigma_k}$ proves $\DNE{(\Sigma_k \lor \Theta)}$. 

3. By Propositions \ref{prop:DUAL_BASIC}.(3) and \ref{prop:DNE_DUAL}, for any $\varphi \in \Sigma_k$ and $\psi \in \Pi_k$, $\HA + \DNE{\Sigma_{k-1}}$ proves $\neg \varphi \leftrightarrow \varphi^\bot$ and $\neg \psi^\bot \leftrightarrow \psi$. 
Thus the principles $\DNE{(\n{\Sigma_k} \lor \Theta)}$ and $\DNE{(\Pi_k \lor \Theta)}$ are equivalent over $\HA + \DNE{\Sigma_{k-1}}$. 
Also by clause 1, $\HA + \DNE{(\Pi_k \lor \Theta)}$ proves $\DNE{\Sigma_{k-1}}$. 

Clause 4 follows from clauses 1 and 3 because $\DNE{\Gamma}$ entails $\DNE{\Sigma_{k-1}}$. 
Clause 5 is proved in a similar way as in the proof of clause 3. 
Clause 6 is a refinement of Proposition \ref{prop:DML_DNS}.(1) in the light of Proposition \ref{prop:DML_DNE}, and is proved in a similar way. 
Clause 7 follows from clause 6 and the fact that $\HA + \DNE{\Gamma}$ proves $\DNS{\Sigma_{k-1}}$. 
Clause 8 is a refinement of Proposition \ref{prop:DML_DNS}.(2). 
Clause 9 follows from clause 8 because $\HA + \DNE{\Gamma}$ proves $\DNE{\Pi_k}$. 
Clause 10 is proved in a similar way as in the proof of clause 3. 
Clause 11 follows from clause 10. 
\end{proof}

We have the following corollary which shows that $\LEM{\Sigma_k}$ and $\LEM{\Pi_k}$ are also variations of the double negation elimination. 
A part of Corollary \ref{cor:DNE2}.(4) is stated in \cite{ABHK}.

\begin{cor}\label{cor:DNE2}\leavevmode
\begin{enumerate}
	\item For $\Gamma' \in \{\Sigma_k, \Delta_k, \n{\Pi_k}, \n{\Delta_k}, \dn{\Sigma_k}, \dn{\Delta_k}\}$, $\DNE{(\Sigma_k \lor \Gamma')}$ is equivalent to $\DNE{\Sigma_k}$ over $\HA$; 
	\item $\LEM{\Sigma_k}$, $\DNE{(\Sigma_k \lor \Pi_k)}$, $\DNE{(\Sigma_k \lor \n{\Sigma_k})}$ and $\DNE{(\Sigma_k \lor \dn{\Pi_k})}$ are equivalent over $\HA$;  
	\item $\LEM{\Pi_k}$, $\DNE{(\n{\Pi_k} \lor \Pi_k)}$ and $\DNE{(\dn{\Sigma_k} \lor \Pi_k)}$ are equivalent over $\HA$; 
	\item $\DML{\Sigma_k}^\bot$, $\DNE{(\Pi_k \lor \Pi_k)}$, $\DNE{(\Pi_k \lor \n{\Sigma_k})}$ and $\DNE{(\Pi_k \lor \dn{\Pi_k})}$ are equivalent over $\HA$; 
	\item Let $\Gamma' \in \{\Sigma_k, \Pi_k, \Delta_k, \n{\Sigma_k}, \n{\Pi_k}\}$ and $\Gamma'' \in \{\Delta_k, \n{\Delta_k}, \dn{\Delta_k}\}$. 
Then $\LEM{\Delta_k}$, $\DNE{(\Delta_k \lor \n{(\Gamma')})}$ and $\DNE{(\Gamma'' \lor \Pi_k)}$ are equivalent over $\HA$; 
	\item $\DNE{(\Delta_k \lor \Delta_k)}$, $\DNE{(\Delta_k \lor \dn{\Delta_k})}$ and $\DML{\n{\Delta_k}} + \DNE{\Sigma_{k-1}}$ are equivalent over $\HA$. 
\end{enumerate}
\end{cor}
\begin{proof}
1. This follows from Proposition \ref{prop:DNE1}.(2). 

2. By Corollary \ref{cor:DML_LEM}.(2), $\LEM{\Sigma_k}$ is equivalent to $\DML{(\n{\Sigma_k}, \n{\Pi_k})} + \DNE{\Sigma_k}$. 
By Proposition \ref{prop:DML_DNE}, it is equivalent to $\DNE{(\dn{\Sigma_k} \lor \dn{\Pi_k})} + \DNE{\Sigma_k}$. 
By Propositions \ref{prop:DNE1}.(5), \ref{prop:DNE1}.(7) and \ref{prop:DNE1}.(9), it is equivalent to $\DNE{(\Sigma_k \lor \Pi_k)}$. 
Also by Propositions \ref{prop:DNE1}.(4) and \ref{prop:DNE1}.(9), each of $\DNE{(\Sigma_k \lor \n{\Sigma_k})}$ and $\DNE{(\Sigma_k \lor \dn{\Pi_k})}$ is equivalent to $\DNE{(\Sigma_k \lor \Pi_k)}$. 

3. By Corollary \ref{cor:DML_LEM}.(1), $\LEM{\Pi_k}$ is equivalent to $\DML{(\n{\Sigma_k}, \n{\Pi_k})} + \DNE{\Sigma_{k-1}}$, and it is equivalent to $\DNE{(\dn{\Sigma_k} \lor \dn{\Pi_k})} + \DNE{\Sigma_{k-1}}$ by Proposition \ref{prop:DML_DNE}. 
By Propositions \ref{prop:DNE1}.(6) and \ref{prop:DNE1}.(9), it is equivalent to $\DNE{(\n{\Pi_k} \lor \Pi_k)}$. 
By Proposition \ref{prop:DNE1}.(7), $\DNE{(\n{\Pi_k} \lor \Pi_k)}$ is equivalent to $\DNE{(\dn{\Sigma_k} \lor \Pi_k)}$.

4. By Corollary \ref{cor:LLPO2}.(1), $\DML{\Sigma_k}^\bot$ is equivalent to $\DML{\Sigma_k} + \DNE{\Sigma_{k-1}}$, and this is equivalent to $\DNE{(\n{\Sigma_k} \lor \n{\Sigma_k})} + \DNE{\Sigma_{k-1}}$. 
Then by Propositions \ref{prop:DNE1}.(3), it is equivalent to $\DNE{(\n{\Sigma_k} \lor \Pi_k)}$. 
It is equivalent to $\DNE{(\Pi_k \lor \Pi_k)}$ by Proposition \ref{prop:DNE1}.(4), and hence, also to $\DNE{(\Pi_k \lor \dn{\Pi_k})}$ by Proposition \ref{prop:DNE1}.(9). 

5. By Corollary \ref{cor:DML_LEM3}, $\LEM{\Delta_k}$ is equivalent to $\DML{(\n{\Delta_k}, \Gamma')} + \DNE{\Sigma_{k-1}}$. 
And it is equivalent to $\DNE{(\dn{\Delta_k} \lor \n{(\Gamma')})} + \DNE{\Sigma_{k-1}}$. 
This is equivalent to $\DNE{(\Delta_k \lor \n{(\Gamma')})}$ by Proposition \ref{prop:DNE1}.(10). 
Also each of $\DNE{(\dn{\Delta_k} \lor \Pi_k)}$ and $\DNE{(\Delta_k \lor \Pi_k)}$ is equivalent to $\DNE{(\Delta_k \lor \n{\Sigma_k})}$ by Propositions \ref{prop:DNE1}.(4) and \ref{prop:DNE1}.(10). 

By Corollary \ref{cor:DML_LEM3}, $\LEM{\Delta_k}$ is equivalent to $\DML{(\Delta_k, \Sigma_k)} + \DNE{\Sigma_{k-1}}$, and it is equivalent to $\DNE{(\n{\Delta_k} \lor \n{\Sigma_k})} + \DNE{\Sigma_{k-1}}$. 
By Proposition \ref{prop:DNE1}.(3), it is equivalent to $\DNE{(\n{\Delta_k} \lor \Pi_k)}$. 

6. This is immediate from Propositions \ref{prop:DML_DNE}, \ref{prop:DNE1}.(10) and \ref{prop:DNE1}.(11). 
\end{proof}



\begin{cor}\label{cor:DLEM}
$\HA + \LEM{\Delta_k} \vdash \DNE{(\Delta_k \lor \Delta_k)}$. 
\end{cor}
\begin{proof}
This is because $\HA + \LEM{\Delta_k} \vdash \DNE{(\Delta_k \lor \Pi_k)}$ by Corollary \ref{cor:DNE2}.(5). 
\end{proof}

In Akama et al.~\cite{ABHK}, it is shown that $\HA + \LEM{\Delta_{k+1}}$ proves $\LEM{\Sigma_k}$. 
The following proposition is a refinement of their result from Corollary \ref{cor:DLEM}. 

\begin{prop}\label{prop:DNE3}
$\HA + \DNE{(\Delta_{k+1} \lor \Delta_{k+1})} \vdash \LEM{\Sigma_k}$. 
\end{prop}
\begin{proof}
Let $\varphi$ be any $\Sigma_k$ formula. 
Since $\HA \vdash \neg(\neg \varphi \land \neg \neg \varphi)$, $\HA + \DNS{\Sigma_{k-1}} \vdash \neg (\neg \varphi \land \neg \varphi^\bot)$ by Proposition \ref{prop:WDUAL}. 
Then $\HA + \DNS{\Sigma_{k-1}} \vdash \neg \neg (\varphi \lor \varphi^\bot)$. 
Since both $\varphi$ and $\varphi^\bot$ are $\Delta_{k+1}$ and $\HA + \DNE{(\Delta_{k+1} \lor \Delta_{k+1})}$ derives $\DNS{\Sigma_{k-1}}$, $\HA + \DNE{(\Delta_{k+1} \lor \Delta_{k+1})} \vdash \varphi \lor \varphi^\bot$. 
Hence the theory proves $\varphi \lor \neg \varphi$ by Proposition \ref{prop:DUAL_BASIC}.(3). 
\end{proof}

Finally, we introduce the following principle based on Peirce's law. 
We show that Peirce's law exactly corresponds to the double negation elimination. 

\begin{defn}[Peirce's law]
Let $\Gamma$ be any set of formulas. 
\begin{tabbing}
\hspace{25mm} \= \hspace{40mm} \= \hspace{50mm} \kill
$\PEIRCE{\Gamma}$ \> $((\varphi \to \psi) \to \varphi) \to \varphi$ \> ($\varphi \in \Gamma$ and $\psi$ is any formula) 
\end{tabbing}
\end{defn}

\begin{prop}
For any set $\Gamma$ of formulas, $\PEIRCE{\Gamma}$ is equivalent to $\DNE{\Gamma}$ over $\HA$. 
\end{prop}
\begin{proof}
First, we prove $\HA + \PEIRCE{\Gamma} \vdash \DNE{\Gamma}$. 
Let $\varphi \in \Gamma$. 
Since $\neg \neg \varphi$ is $(\varphi \to \bot) \to \bot$, $\HA \vdash \neg \neg \varphi \to ((\varphi \to \bot) \to \varphi)$. 
Thus $\HA + \PEIRCE{\Gamma} \vdash \neg \neg \varphi \to \varphi$. 

Secondly, we prove $\HA + \DNE{\Gamma} \vdash \PEIRCE{\Gamma}$. 
Let $\varphi$ be any $\Gamma$ formula and $\psi$ be arbitrary formula. 
Since $\HA$ proves $\neg \varphi \to (\varphi \to \psi)$, $\HA$ also proves $((\varphi \to \psi) \to \varphi) \land \neg \varphi \to \varphi$. 
Hence $\HA \vdash ((\varphi \to \psi) \to \varphi) \to \neg \neg \varphi$.
We obtain $\HA + \DNE{\Gamma} \vdash ((\varphi \to \psi) \to \varphi) \to \varphi$.
\end{proof}

We get the table which summarizes principles equivalent to $\DNE{(\Gamma \lor \Theta)}$ over the theory $\HA + \DNS{\Sigma_{k-1}}$. 
Notice that from Propositions \ref{prop:DNE1}.(6) and \ref{prop:DNE1}.(8), $\DNE{(\dn{\Sigma_k} \lor  \Theta)}$ and $\DNE{(\dn{\Pi_k} \lor  \Theta)}$ are equivalent to $\DNE{(\n{\Pi_k} \lor  \Theta)}$ and $\DNE{(\n{\Sigma_k} \lor  \Theta)}$ over $\HA + \DNS{\Sigma_{k-1}}$, respectively. 
So $\dn{\Sigma_k}$ and $\dn{\Pi_k}$ are excluded from the table. 

\begin{table}[h]
\centering
\begin{tabular}{|c||c|c|c|c|}
\hline
\backslashbox{$\Gamma$}{$\Theta$} & $\Sigma_k$ & $\n{\Pi_k}$ & $\Pi_k$ & $\n{\Sigma_k}$ \\

\hline
\hline

$\Sigma_k$ & $\DNE{\Sigma_k}$ & $\DNE{\Sigma_k}$ & $\LEM{\Sigma_k}$ & $\LEM{\Sigma_k}$ \\

\hline

$\n{\Pi_k}$ &  & $\DML{\Pi_k}$ & $\LEM{\Pi_k}$ & $\LEM{\n{\Sigma_k}}$ \\

\hline

$\Pi_k$ & & & $\DML{\Sigma_k}^\bot$ & $\DML{\Sigma_k}^\bot$ \\

\hline 

$\n{\Sigma_k}$ & & & & $\DML{\Sigma_k}$ \\

\hline 
\end{tabular}

\vspace{0.1in}

\begin{tabular}{|c||c|c|c|}
\hline
\backslashbox{$\Gamma$}{$\Theta$} & $\Delta_k$ & $\n{\Delta_k}$ & $\dn{\Delta_k}$ \\

\hline
\hline

$\Sigma_k$ & $\DNE{\Sigma_k}$ & $\DNE{\Sigma_k}$ & $\DNE{\Sigma_k}$ \\

\hline

$\n{\Pi_k}$ & $\LEM{\Delta_k}$ & $\LEM{\n{\Delta_k}}$ & $\LEM{\n{\Delta_k}}$ \\

\hline

$\Pi_k$ & $\LEM{\Delta_k}$ & $\LEM{\Delta_k}$ & $\LEM{\Delta_k}$\\

\hline 

$\n{\Sigma_k}$ & $\LEM{\Delta_k}$ & $\LEM{\n{\Delta_k}}$ & $\LEM{\n{\Delta_k}}$\\

\hline

$\Delta_k$ & $\DNE{(\Delta_k \lor \Delta_k)}$ & $\LEM{\Delta_k}$ & $\DNE{(\Delta_k \lor \Delta_k)}$ \\

\hline

$\n{\Delta_k}$ & & $\DML{\Delta_k}$ & $\LEM{\n{\Delta_k}}$ \\

\hline

$\dn{\Delta_k}$ & & & $\DML{\n{\Delta_k}}$ \\

\hline
\end{tabular}
\caption{Principles equivalent to $\DNE{(\Gamma \lor \Theta)}$ over $\HA + \DNS{\Sigma_{k-1}}$}
\end{table}

\section{The constant domain axiom}\label{section:CD}

In this section, we investigate the principles of the form $\CD{(\Gamma, \Theta)}$ in Definition \ref{defn:principles2}, and classify them in the arithmetical hierarchy of classical principles. 
Note that $\CD{(\Gamma, \Theta)}$ is not equivalent to $\CD{(\Theta, \Gamma)}$ in general. 

In first-order intuitionistic Kripke semantics, the constant domain axiom corresponds to Kripke frames with constant domains (cf.~\cite[p.~328]{Smor}). 
First of all, we show that in our framework of first-order intuitionistic arithmetic, the constant domain axiom is equivalent to the law of excluded middle despite its semantic origin. 
Let $\mathbf{LEM}$ and $\mathbf{CD}$ denote the principles $\LEM{\Fml}$ and $\CD{(\Fml, \Fml)}$ respectively, where $\Fml$ is the set of all formulas. 

\begin{prop}\label{prop:CD_LEM}
$\mathbf{CD}$ is equivalent to $\mathbf{LEM}$ over $\HA$. 
\end{prop}
\begin{proof}
First, we prove $\HA + \mathbf{CD} \vdash \varphi \lor \neg \varphi$ for any formula $\varphi$ by induction on the construction of $\varphi$. 
If $\varphi$ is an atomic formula, then the statement is obvious. 

Assume that $\HA + \mathbf{CD}$ proves $\psi \lor \neg \psi$ and $\rho \lor \neg \rho$, and suppose $\varphi$ is one of the forms $\psi \land \rho$, $\psi \lor \rho$ and $\psi \to \rho$. 
Notice that $\neg \psi \lor \neg \rho \to \neg(\psi \land \rho)$, $\neg \psi \land \neg \rho \to \neg (\psi \lor \rho)$, $\neg \psi \lor \rho \to (\psi \to \rho)$ and $\psi \land \neg \rho \to \neg (\psi \to \rho)$ are provable in $\HA$. 
Therefore $\varphi \lor \neg \varphi$ is also provable in $\HA + \mathbf{CD}$. 

Assume that $\HA + \mathbf{CD}$ proves $\psi(x) \lor \neg \psi(x)$. 
Then $\forall x (\exists x \psi(x) \lor \neg \psi(x))$ and $\forall x (\psi(x) \lor \exists x\, \neg \psi(x))$ are also provable. 
By applying $\mathbf{CD}$, we obtain that $\HA + \mathbf{CD}$ proves $\exists x \psi(x) \lor \neg\, \exists x \psi(x)$ and $\forall x \psi(x) \lor \neg\, \forall x \psi(x)$. 
Therefore, if $\varphi$ is of one of the forms $\exists x \psi(x)$ and $\forall x \psi(x)$, then $\varphi \lor \neg \varphi$ is provable in $\HA + \mathbf{CD}$. 

Secondly, we prove $\HA + \mathbf{LEM} \vdash \mathbf{CD}$. 
Let $\varphi$ and $\psi(x)$ be any formulas with $x \notin \FV(\varphi)$. 
We have $\HA \vdash \forall x(\varphi \lor \psi(x)) \land \neg \varphi \to \forall x \psi(x)$. 
Since $\HA + \mathbf{LEM}$ proves $\varphi \lor \neg \varphi$, we conclude that $\HA + \mathbf{LEM}$ also proves $\forall x(\varphi \lor \psi(x)) \to \varphi \lor \forall x \psi(x)$. 
\end{proof}

\begin{prop}\label{prop:CD0}\leavevmode
\begin{enumerate}
	\item $\CD{(\Gamma, \Pi_{k+1})}$ is equivalent to $\CD{(\Gamma, \Sigma_k)}$ over $\HA$; 
	\item $\CD{(\Gamma, \n{\Sigma_{k+1}})}$ is equivalent to $\CD{(\Gamma, \n{\Pi_k})}$ over $\HA$. 
\end{enumerate}
\end{prop}
\begin{proof}
These statements are proved by using a primitive recursive pairing function. 
\end{proof}

As in the proof of Proposition \ref{prop:CD_LEM}, we can show that $\LEM{\Gamma}$ and $\LEM{\Delta_k}$ are sufficiently strong for the constant domain axiom. 

\begin{prop}\label{prop:CD1}
Let $\Gamma$ and $\Theta$ be any sets of formulas. 
\begin{enumerate}
	\item $\HA + \LEM{\Gamma} \vdash \CD{(\Gamma, \Theta)}$; 
	\item $\HA + \LEM{\Delta_k} \vdash \CD{(\Delta_k, \Theta)}$. 
\end{enumerate}
\end{prop}
%

From the prenex normal form theorem proved in \cite[Theorem 2.7]{ABHK} and \cite[Theorem 5.7]{FuKu}, $\mathbf{LEM}$ is equivalent to $\bigcup\{\LEM{\Sigma_k} \mid k \geq 0\}$ over $\HA$. 
Therefore, the following proposition can be regarded as a stratification of Proposition \ref{prop:CD_LEM}. 

\begin{prop}\label{prop:CDS1}
Let $\Theta$ be a set of formulas such that $\Sigma_{k-1} \subseteq \Theta$. 
Then the following are equivalent over $\HA$: 
\begin{enumerate}
	\item $\CD{(\Sigma_k, \Theta)}$. 
	\item $\LEM{\Sigma_k}$. 
\end{enumerate}
\end{prop}
\begin{proof}
First, we prove $\HA + \CD{(\Sigma_k, \Sigma_{k-1})} \vdash \LEM{\Sigma_k}$ by induction on $k$. 
For $k=0$, the statement is trivial. 
Suppose that the statement holds for $k$, and we prove $\HA + \CD{(\Sigma_{k+1}, \Sigma_k)} \vdash \LEM{\Sigma_{k+1}}$. 
Let $\exists x \varphi(x)$ be any $\Sigma_{k+1}$ formula with $\varphi(x) \in \Pi_k$. 
By induction hypothesis and Fact \ref{fact:ABHK}, $\HA + \CD{(\Sigma_k, \Sigma_{k-1})}$ proves $\LEM{\Pi_k} + \DNE{\Sigma_k}$. 
Thus $\HA + \CD{(\Sigma_k, \Sigma_{k-1})} \vdash \varphi(x) \lor \neg \varphi(x)$. 
We get $\HA + \CD{(\Sigma_k, \Sigma_{k-1})} \vdash \forall x (\exists x \varphi(x) \lor \varphi^\bot(x))$ by using $\DUAL{\Pi_k}$. 
Then
\[
	\HA + \CD{(\Sigma_{k+1}, \Sigma_k)} \vdash \exists x \varphi(x) \lor \forall x \varphi^\bot(x). 
\]
This implies $\HA + \CD{(\Sigma_{k+1}, \Sigma_k)} \vdash \exists x \varphi(x) \lor \neg\, \exists x \varphi(x)$. 

On the other hand, $\HA + \LEM{\Sigma_k} \vdash \CD{(\Sigma_k, \Theta)}$ follows from Proposition \ref{prop:CD1}.(1).
\end{proof}

Fact \ref{FuKaCD} states that $\CD{(\Pi_1, \Pi_1)}$ is $\HA$-equivalent to $\DML{\Sigma_1}$. 
By Corollary \ref{cor:LLPO2}.(1), $\DML{\Sigma_1}$ is $\HA$-equivalent to $\DML{\Sigma_1}^\bot$. 
So the following proposition is a generalization of Fact \ref{FuKaCD}. 

\begin{prop}\label{prop:CDP1}
The following are equivalent over $\HA$: 
\begin{enumerate}
	\item $\CD{(\Pi_k, \Pi_k)}$. 
	\item $\DML{\Sigma_k}^\bot$. 
\end{enumerate}
\end{prop}
\begin{proof}
First, we prove $\HA + \DML{\Sigma_k}^\bot  \vdash \CD{(\Pi_k, \Pi_k)}$. 
Let $\varphi, \psi(x) \in \Pi_k$ with $x \notin \FV(\varphi)$. 
Since $\HA \vdash \forall x(\varphi \lor \psi(x)) \land \neg \varphi \to \forall x \psi(x)$, $\HA$ proves $\forall x(\varphi \lor \psi(x)) \to \neg (\neg \varphi \land \neg\, \forall x \psi(x))$. 
By Proposition \ref{prop:DUAL_BASIC}.(3), $\HA \vdash \forall x(\varphi \lor \psi(x)) \to \neg (\varphi^\bot \land (\forall x \psi(x))^\bot)$.
Then we obtain
\[
	\HA + \DML{\Sigma_k}^\bot \vdash \forall x(\varphi \lor \psi(x)) \to \varphi^{\bot \bot} \lor (\forall x\psi(x))^{\bot \bot}. 
\]
By Proposition \ref{prop:DUAL_BASIC}.(2), we conclude
\[
	\HA + \DML{\Sigma_k}^\bot \vdash \forall x(\varphi \lor \psi(x)) \to \varphi \lor \forall x\psi(x). 
\]

Secondly, we prove $\HA + \CD{(\Pi_k, \Pi_k)} \vdash \DML{\Sigma_k}^\bot$. 
We may assume $k > 0$. 
Let $\exists x \varphi(x)$ and $\exists y \psi(y)$ be any $\Sigma_k$ formulas where $\varphi(x)$ and $\psi(y)$ are $\Pi_{k-1}$. 
Since $\psi(y)$ implies $\exists y \psi(y)$, we obtain
\begin{equation}\label{eqCD1}
	\HA \vdash \neg(\exists x \varphi(x) \land \exists y \psi(y)) \land \psi(y) \to \neg\, \exists x \varphi(x). 
\end{equation}
Since $\HA + \CD{(\Pi_k, \Pi_k)}$ entails $\CD{(\Sigma_{k-1}, \Pi_k)}$, we obtain that $\HA + \CD{(\Pi_k, \Pi_k)}$ proves $\LEM{\Pi_{k-1}} + \DNE{\Sigma_{k-1}}$ by Proposition \ref{prop:CDS1} and Fact \ref{fact:ABHK}. 
Hence $\HA + \CD{(\Pi_k, \Pi_k)} \vdash \psi(y) \lor \neg \psi(y)$. 
From this with (\ref{eqCD1}), we have
\[
	\HA + \CD{(\Pi_k, \Pi_k)} \vdash \neg(\exists x \varphi(x) \land \exists y \psi(y)) \to \forall y(\neg\, \exists x \varphi(x) \lor \neg \psi(y)). 
\]
By using $\DUAL{\Sigma_k}$, we get
\[
	\HA + \CD{(\Pi_k, \Pi_k)} \vdash \neg(\exists x \varphi(x) \land \exists y \psi(y)) \to \forall y((\exists x \varphi(x))^\bot \lor \psi^\bot(y)). 
\]
Since $(\exists x \varphi(x))^\bot \in \Pi_k$ and $\psi^\bot(y) \in \Sigma_{k-1}$, we obtain
\[
	\HA + \CD{(\Pi_k, \Pi_k)} \vdash \neg(\exists x \varphi(x) \land \exists y \psi(y)) \to (\exists x \varphi(x))^\bot \lor \forall y \psi^\bot(y). 
\]
Therefore
\[
	\HA + \CD{(\Pi_k, \Pi_k)} \vdash \neg(\exists x \varphi(x) \land \exists y \psi(y)) \to (\exists x \varphi(x))^\bot \lor (\exists y \psi(y))^\bot. \qedhere
\]
\end{proof}

From Corollaries \ref{cor:LLPO2}.(1) and \ref{cor:DNE2}.(4) and Propositions \ref{prop:Coll3}, \ref{prop:Coll4} and \ref{prop:CDP1}, we have the following result. 

\begin{cor}\label{cor:MC}
For $k \geq 1$, the following are equivalent over $\HA$:
\begin{enumerate}
	\item $\DML{\Sigma_k} + \DNE{\Sigma_{k-1}}$. 
	\item $\DML{\Sigma_k}^\bot$. 
	\item $\CD{(\Pi_k, \Pi_k)}$. 
	\item $\COLL{\Pi_k}$. 
	\item $\DNE{(\Pi_k \lor \Pi_k)}$. 
\end{enumerate}
\end{cor}

\begin{cor}\label{cor:MC2}
For $k \geq 1$, each of $\DML{\Sigma_k}^\bot$, $\CD{(\Pi_k, \Pi_k)}$, $\COLL{\Pi_k}$ and $\DNE{(\Pi_k \lor \Pi_k)}$ implies $\DML{\Pi_k}$ over $\HA$. 
\end{cor}
\begin{proof}
This is immediate from Proposition \ref{prop:PSDML} and Corollary \ref{cor:MC}. 
\end{proof}

\begin{prop}\label{prop:CDD1}
Let $\Theta$ be a set of formulas such that $\Sigma_{k-1} \subseteq \Theta$. 
Then the following are equivalent over $\HA$: 
\begin{enumerate}
	\item $\CD{(\Delta_k, \Theta)}$. 
	\item $\LEM{\Delta_k}$. 
\end{enumerate}
\end{prop}
\begin{proof}
Notice that $\CD{(\Delta_k, \Sigma_{k-1})}$ implies $\CD{(\Sigma_{k-1}, \Sigma_{k-1})}$. 
Then by Proposition \ref{prop:CDS1} and Fact \ref{fact:ABHK}, $\HA + \CD{(\Delta_k, \Sigma_{k-1})}$ proves $\LEM{\Pi_{k-1}} + \DNE{\Sigma_{k-1}}$. 
Therefore the statement $\HA + \CD{(\Delta_k, \Sigma_{k-1})} \vdash \LEM{\Delta_k}$ is proved as in the proof of Proposition \ref{prop:CDS1}. 
On the other hand, $\HA + \LEM{\Delta_k} \vdash \CD{(\Delta_k, \Theta)}$ follows from Proposition \ref{prop:CD1}.(2). 
\end{proof}

Next, we investigate the principles $\CD{(\n{\Gamma}, \Theta)}$ and $\CD{(\n{\Delta_k}, \Theta)}$. 
In the light of Proposition \ref{prop:CD1}, they are derived from $\LEM{\n{\Gamma}}$ and $\LEM{\n{\Delta_k}}$, respectively. 
In addition, for $\Theta = \n{\Sigma_k}$, we obtain the following proposition. 

\begin{prop}\label{prop:CD2}Let $\Gamma$ be any set of formulas. 
\begin{enumerate}
	\item $\HA + \DML{(\Gamma, \Sigma_k)} \vdash \CD{(\n{\Gamma}, \n{\Sigma_k})}$; 
	\item $\HA + \DML{(\Delta_k, \Sigma_k)} \vdash \CD{(\n{\Delta_k}, \n{\Sigma_k})}$.
\end{enumerate}
\end{prop}
\begin{proof}
1. By Proposition \ref{prop:CD0}.(2), it suffices to show that $\HA + \DML{(\Gamma, \Sigma_k)}$ proves $\CD{(\n{\Gamma}, \n{\Pi_{k-1}})}$. 
Let $\varphi \in \Gamma$ and $\psi(x) \in \Pi_{k-1}$ with $x \notin \FV(\varphi)$. 
Then we have
\begin{align*}
	\HA \vdash \forall x(\neg \varphi \lor \neg \psi(x)) & \to \forall x\, \neg (\varphi \land \psi(x)), \\
	& \to \neg\, \exists x (\varphi \land \psi(x)), \\
	& \to \neg (\varphi \land \exists x \psi(x)). 
\end{align*}
Thus
\[
	\HA + \DML{(\Gamma, \Sigma_k)} \vdash \forall x(\neg \varphi \lor \neg \psi(x)) \to \neg \varphi \lor \neg\, \exists x \psi(x). 
\]
We conclude
\[
	\HA + \DML{(\Gamma, \Sigma_k)} \vdash \forall x(\neg \varphi \lor \neg \psi(x)) \to \neg \varphi \lor \forall x\, \neg \psi(x). 
\]

2 is proved similarly. 
\end{proof}

With the help of $\DNS{\Sigma_{k-2}}$, the converse implications also hold.

\begin{prop}\label{prop:CD3}\leavevmode
\begin{enumerate}
	\item $\HA + \CD{(\n{\Pi_k}, \n{\Sigma_k})} + \DNS{\Sigma_{k-2}} \vdash \DML{(\Sigma_k, \Pi_k)}$; 
	\item $\HA + \CD{(\n{\Sigma_k}, \n{\Sigma_k})} + \DNS{\Sigma_{k-2}} \vdash \DML{\Sigma_k}$; 
	\item $\HA + \CD{(\n{\Delta_k}, \n{\Sigma_k})} + \DNS{\Sigma_{k-2}} \vdash \DML{(\Delta_k, \Sigma_k)}$.  
\end{enumerate}
\end{prop}
\begin{proof}
1. We prove by induction on $k \geq 0$. 
The statement for $k=0$ is trivial. 
We assume that our statement holds for $k$, and we prove $\HA + \CD{(\n{\Pi_{k+1}}, \n{\Sigma_{k+1}})} + \DNS{\Sigma_{k-1}} \vdash \DML{(\Sigma_{k+1}, \Pi_{k+1})}$. 
Let $\exists x \varphi(x) \in \Sigma_{k+1}$ and $\psi \in \Pi_{k+1}$ where $\varphi(x) \in \Pi_k$. 
We have 
\begin{align*}
	\HA \vdash \neg (\exists x \varphi(x) \land \psi) & \to \neg\, \exists x (\varphi(x) \land \psi), \\
	& \to \forall x\, \neg (\varphi(x) \land \psi), \\
	& \to \forall x\, \neg (\neg \neg \varphi(x) \land \neg \neg \psi). 
\end{align*}
Then
\begin{equation}\label{eqCD2}
	\HA \vdash \neg (\exists x \varphi(x) \land \psi) \land \neg \neg \varphi(x) \to \neg \psi.
\end{equation}

By induction hypothesis, $\HA + \CD{(\n{\Pi_k}, \n{\Sigma_k})} + \DNS{\Sigma_{k-2}} \vdash \DML{(\Sigma_k, \Pi_k)}$. 
By Corollary \ref{cor:DML2}.(1), $\HA + \CD{(\n{\Pi_k}, \n{\Sigma_k})} + \DNS{\Sigma_{k-1}}$ proves $\LEM{\n{\Pi_k}}$. 
Thus we have that $\HA + \CD{(\n{\Pi_k}, \n{\Sigma_k})} + \DNS{\Sigma_{k-1}}$ proves $\neg \neg \varphi(x) \lor \neg \varphi(x)$. 
From this with (\ref{eqCD2}), we obtain
\[
	\HA + \CD{(\n{\Pi_k}, \n{\Sigma_k})} + \DNS{\Sigma_{k-1}} \vdash \neg (\exists x \varphi(x) \land \psi) \to \forall x(\neg \psi \lor \neg \varphi(x)). 
\]
By applying $\CD{(\n{\Pi_{k+1}}, \n{\Sigma_{k+1}})}$, we have
\[
	\HA + \CD{(\n{\Pi_{k+1}}, \n{\Sigma_{k+1}})} + \DNS{\Sigma_{k-1}} \vdash \neg (\exists x \varphi(x) \land \psi) \to \neg \psi \lor \forall x\, \neg \varphi(x). 
\]
We conclude
\[
	\HA + \CD{(\n{\Pi_{k+1}}, \n{\Sigma_{k+1}})} + \DNS{\Sigma_{k-1}} \vdash \neg (\exists x \varphi(x) \land \psi) \to \neg\, \exists x \varphi(x) \lor \neg \psi. 
\]

2. We may assume $k > 0$. 
Let $\exists x \varphi(x)$ and $\exists y \psi(y)$ be any $\Sigma_k$ formulas with $\varphi(x), \psi(y) \in \Pi_{k-1}$. 
\begin{align*}
	\HA \vdash \neg(\exists x \varphi(x) \land \exists y \psi(y)) & \to \neg\, \exists x \exists y (\varphi(x) \land \psi(y)), \\
	& \to \forall x \forall y\, \neg (\varphi(x) \land \psi(y)). 
\end{align*}
Since $\CD{(\n{\Sigma_k}, \n{\Sigma_k})}$ entails $\CD{(\n{\Pi_{k-1}}, \n{\Sigma_{k-1}})}$,  by clause 1, we have that $\HA + \CD{(\n{\Sigma_k}, \n{\Sigma_k})} + \DNS{\Sigma_{k-3}}$ proves $\DML{(\Sigma_{k-1}, \Pi_{k-1})}$. 
Then by Corollary \ref{cor:DML2}.(1), $\HA + \CD{(\n{\Sigma_k}, \n{\Sigma_k})} + \DNS{\Sigma_{k-2}}$ proves $\LEM{\n{\Pi_{k-1}}}$. 
By Proposition \ref{prop:DML1}.(1), it also proves $\DML{\Pi_{k-1}}$. 
Thus
\[
	\HA + \CD{(\n{\Sigma_k}, \n{\Sigma_k})} + \DNS{\Sigma_{k-2}} \vdash \neg(\exists x \varphi(x) \land \exists y \psi(y)) \to \forall x \forall y (\neg \varphi(x) \lor \neg \psi(y)). 
\]
By applying $\CD{(\n{\Sigma_k}, \n{\Sigma_k})}$ twice, we obtain
\[
	\HA + \CD{(\n{\Sigma_k}, \n{\Sigma_k})} + \DNS{\Sigma_{k-2}} \vdash \neg(\exists x \varphi(x) \land \exists y \psi(y)) \to \forall x\, \neg \varphi(x) \lor \forall y\, \neg \psi(y). 
\]
We conclude
\[
	\HA + \CD{(\n{\Sigma_k}, \n{\Sigma_k})} + \DNS{\Sigma_{k-2}} \vdash \neg(\exists x \varphi(x) \land \exists y \psi(y)) \to \neg\, \exists x \varphi(x) \lor \neg\, \exists y \psi(y). 
\]

3 is proved as in the proof of clause 2. 
\end{proof}

We obtain the following corollary. 

\begin{cor}\leavevmode
Let $\Theta$ be any set of formulas such that $\n{\Pi_{k-1}} \subseteq \Theta$. 
\begin{enumerate}
	\item $\CD{(\n{\Pi_k}, \n{\Sigma_k})}$ is equivalent to $\DML{(\Sigma_k, \Pi_k)}$ over $\HA + \DNS{\Sigma_{k-2}}$; 
	\item $\CD{(\n{\Pi_k}, \Theta)}$ is equivalent to $\LEM{\n{\Pi_k}}$ over $\HA + \DNS{\Sigma_{k-1}}$; 
	\item $\CD{(\n{\Sigma_k}, \n{\Sigma_k})}$ is equivalent to $\DML{\Sigma_k}$ over $\HA + \DNS{\Sigma_{k-2}}$; 
	\item $\CD{(\n{\Delta_k}, \n{\Sigma_k})}$ is equivalent to $\DML{(\Delta_k, \Sigma_k)}$ over $\HA + \DNS{\Sigma_{k-2}}$; 
	\item $\CD{(\n{\Delta_k}, \Theta)}$ is equivalent to $\LEM{\n{\Delta_k}}$ over $\HA + \DNS{\Sigma_{k-1}}$. 
\end{enumerate}
\end{cor}
\begin{proof}
1. This is immediate from Propositions \ref{prop:CD2}.(1) and \ref{prop:CD3}.(1). 

2. From clause 1, Proposition \ref{prop:CD0} and Corollary \ref{cor:DML2}.(1), we have that $\HA + \CD{(\n{\Pi_k}, \n{\Pi_{k-1}})} + \DNS{\Sigma_{k-1}}$ proves $\LEM{\n{\Pi_k}}$. 
On the other hand, $\HA + \LEM{\n{\Pi_k}}$ proves $\CD{(\n{\Pi_k}, \Theta)}$ by Proposition \ref{prop:CD1}.(1). 

3. This is a consequence of Propositions \ref{prop:CD2}.(1) and \ref{prop:CD3}.(2). 

4. Immediate from Propositions \ref{prop:CD2}.(2) and \ref{prop:CD3}.(3). 

5. As in the proof of clause 2, we obtain the statement from clause 4, Propositions \ref{prop:CD0}, \ref{prop:CD1}.(2) and Corollary \ref{cor:DML2}.(4),
\end{proof}

\begin{prob}\leavevmode
\begin{itemize}
	\item Is there a set $\Theta$ of formulas such that $\HA + \CD{(\Pi_k, \Theta)}$ proves $\LEM{\Pi_k}$?
	\item Is there a set $\Theta$ of formulas such that $\HA + \CD{(\n{\Sigma_k}, \Theta)} + \DNS{\Sigma_{k-1}}$ proves $\LEM{\n{\Sigma_k}}$?
\end{itemize}
\end{prob}

The following figure (Figure \ref{fig:CD}) summarizes the situation for implications around the constant domain axioms for negated formulas. 
In \cite[Example 10]{FINSY}, it is shown that $\HA + \DML{\Sigma_k} + \DNE{\Sigma_k}$ does not prove $\LEM{\n{\Sigma_k}}$ for $k \geq 1$. 
Therefore, in Figure \ref{fig:CD}, $\DML{\Sigma_k}$ does not imply $\LEM{\n{\Sigma_k}}$ even in the theory $\HA + \DNE{\Sigma_k}$ for $k \geq 1$. 

\begin{figure}[ht]
\centering
\begin{tikzpicture}
\node (NDNSCD) at (0,0) {$\CD{(\n{\Delta_k}, \n{\Sigma_k})}$};
\node (NSNSCD) at (-2, 1) {$\CD{(\n{\Sigma_k}, \n{\Sigma_k})}$};
\node (NPNSCD) at (2, 1) {$\CD{(\n{\Pi_k}, \n{\Sigma_k})}$};
\node (DSDML) at (0, 2) {$\DML{(\Delta_k, \Sigma_k)}$};
\node (SDML) at (-2, 3) {$\DML{\Sigma_k}$};
\node (SPDML) at (2, 3) {$\DML{(\Sigma_k, \Pi_k)}$};
\node (NDLEM) at (0, 4) {$\LEM{\n{\Delta_k}}$};
\node (NSLEM) at (-2, 5) {$\LEM{\n{\Sigma_k}}$};
\node (NPLEM) at (2, 5) {$\LEM{\n{\Pi_k}}$};
\node (NSTCD) at (-5, 2) {$\CD{(\n{\Sigma_k}, \Theta)}$};
\node (NPTCD) at (5, 2) {$\CD{(\n{\Pi_k}, \Theta)}$};

\draw [<-] (NDNSCD)--(DSDML);
\draw [<-] (NDNSCD)--(NSNSCD);
\draw [<-] (NDNSCD)--(NPNSCD);
\draw [<-] (DSDML)--(SDML);
\draw [<-] (DSDML)--(SPDML);
\draw [<-] (NSNSCD)--(SDML);
\draw [<-] (NPNSCD)--(SPDML);
\draw [<-] (DSDML)--(NDLEM);
\draw [<-] (NDLEM)--(NSLEM);
\draw [<-] (NDLEM)--(NPLEM);
\draw [<-] (SDML)--(NSLEM);
\draw [<-] (SPDML) to [out=180, in=-50] (NSLEM);
\draw [<-] (SPDML)--(NPLEM);
\draw [<-] (NSNSCD)--(NSTCD);
\draw [<-] (NPNSCD)--(NPTCD);
\draw [<-] (NSTCD)--(NSLEM);
\draw [<-] (NPTCD)--(NPLEM);

\draw [->, dashed] (NDNSCD) to [out=120, in=240] (DSDML);
\draw [->, dashed] (NSNSCD) to [out=120, in=240] (SDML);
\draw [->, dashed] (NPNSCD) to [out=120, in=240] (SPDML);

\draw [->, dotted] (SPDML) to [out=180, in=-75] (NSLEM);
\draw [->, dotted] (SPDML) to [out=120, in=240] (NPLEM);
\draw [->, dotted] (DSDML) to [out=120, in=240] (NDLEM);

\node at (-4, -1) [right]{$\Theta$: A sufficiently large set of formulas};
\draw [->, dashed] (-4, -1.5)--(-3, -1.5) node[right]{: Implication in $\HA + \DNS{\Sigma_{k-2}}$};
\draw [->, dotted] (-4, -2)--(-3, -2) node[right]{: Implication in $\HA + \DNS{\Sigma_{k-1}}$};

\end{tikzpicture}
\caption{Implications around the constant domain axioms for negated formulas}\label{fig:CD}
\end{figure}

\section{Summary}\label{section:rem}

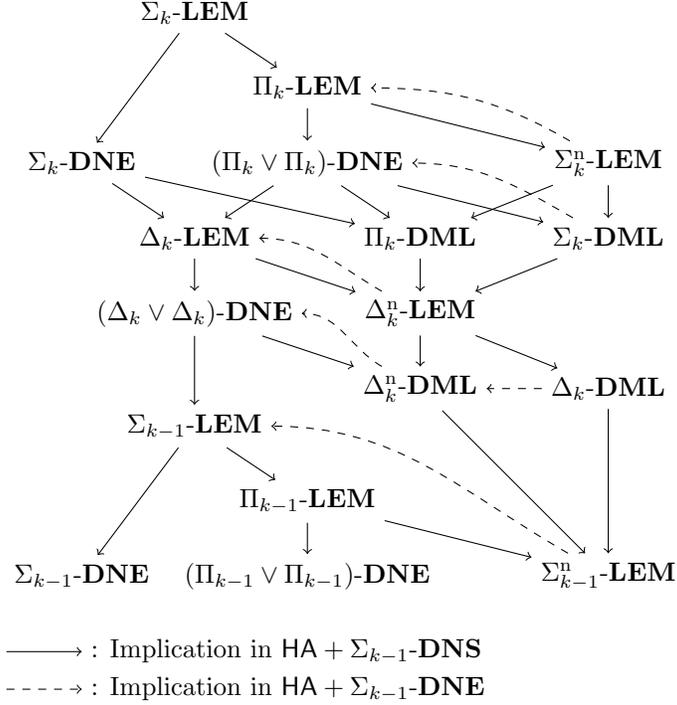
\begin{figure}[ht]
\centering
\begin{tikzpicture}
\node (S-DNE) at (-3, -0.5) {$\DNE{\Sigma_{k-1}}$};
\node (PP-DNE) at (0, -0.5) {$\DNE{(\Pi_{k-1} \lor \Pi_{k-1})}$};
\node (NS-LEM) at (4, -0.5) {$\LEM{\n{\Sigma_{k-1}}}$};
\node (P-LEM) at (0, 0.5) {$\LEM{\Pi_{k-1}}$};
\node (S-LEM) at (-1.5, 1.5) {$\LEM{\Sigma_{k-1}}$};
\node (NDDML) at (1.5,2) {$\DML{\n{\Delta_k}}$};
\node (DDML) at (4,2) {$\DML{\Delta_k}$};
\node (DDDNE) at (-1.5,3)  {$\DNE{(\Delta_k \lor \Delta_k)}$};
\node (NDLEM) at (1.5,3)  {$\LEM{\n{\Delta_k}}$};
\node (DLEM) at (-1.5,4) {$\LEM{\Delta_k}$};
\node (PDML) at (1.5,4) {$\DML{\Pi_k}$};
\node (SDML) at (4,4) {$\DML{\Sigma_k}$};
\node (SDNE) at (-3,5) {$\DNE{\Sigma_k}$};
\node (PPDNE) at (0,5) {$\DNE{(\Pi_k \lor \Pi_k)}$};
\node (NSLEM) at (4,5) {$\LEM{\n{\Sigma_k}}$};
\node (PLEM) at (0,6) {$\LEM{\Pi_k}$};
\node (SLEM) at (-1.5,7) {$\LEM{\Sigma_k}$};

\draw [->] (SLEM)--(PLEM);
\draw [->] (SLEM)--(SDNE);
\draw [->] (PLEM)--(PPDNE);
\draw [->] (PLEM)--(NSLEM);
\draw [->] (SDNE)--(DLEM);
\draw [->] (SDNE)-- (PDML);
\draw [->] (PPDNE)--(DLEM);
\draw [->] (PPDNE)--(PDML);
\draw [->] (PPDNE)--(SDML);
\draw [->] (NSLEM)--(SDML);
\draw [->] (DLEM)--(DDDNE);
\draw [->] (DLEM)--(NDLEM);
\draw [->] (DDDNE)--(S-LEM);
\draw [->] (DDDNE)--(NDDML);
\draw [->] (NDLEM)--(NDDML);
\draw [->] (NDLEM)--(DDML);
\draw [->] (S-LEM)--(S-DNE);
\draw [->] (S-LEM)--(P-LEM);
\draw [->] (P-LEM)--(PP-DNE);
\draw [->] (P-LEM)--(NS-LEM);

\draw [->] (NSLEM)--(PDML);
\draw [->] (PDML)--(NDLEM);
\draw [->] (SDML)--(NDLEM);
\draw [->] (NDDML)--(NS-LEM);
\draw [->] (DDML)--(NS-LEM);

\draw [->, dashed] (NSLEM) to [out=150, in=0] (PLEM);
\draw [->, dashed] (SDML) to [out=150, in=0] (PPDNE);
\draw [->, dashed] (NDLEM) to [out=150, in=0] (DLEM);
\draw [->, dashed] (NDDML) to [out=150, in=0] (DDDNE);
\draw [->, dashed] (NS-LEM) to [out=150, in=0] (S-LEM);
\draw [->, dashed] (DDML)--(NDDML);

\draw [->] (-4, -1.5)--(-3, -1.5) node[right]{: Implication in $\HA + \DNS{\Sigma_{k-1}}$};
\draw [->, dashed] (-4, -2)--(-3, -2) node[right]{: Implication in $\HA + \DNE{\Sigma_{k-1}}$};

\end{tikzpicture}
\caption{A refined arithmetical hierarchy of classical principles}\label{fig:DNS}
\end{figure}

\begin{table}[ht]
\begin{tabular}{|l|l|l|}
\hline
Implications & Verifying theories & cf. \\
\hline
\hline
$\LEM{\Sigma_k} \to \LEM{\Pi_k}$ & $\HA$ & Fact \ref{fact:ABHK}.(1) \\
\hline
$\LEM{\Sigma_k} \to \DNE{\Sigma_k}$ & $\HA$ & Fact \ref{fact:ABHK}.(1) \\
\hline
$\LEM{\Pi_k} \to \DNE{(\Pi_k \lor \Pi_k)}$ & $\HA$ & Fact \ref{fact:ABHK}.(2) \\
\hline
$\LEM{\Pi_k} \to \LEM{\n{\Sigma_k}}$ & $\HA + \DNS{\Sigma_{k-1}}$ & Propositions \ref{prop:NLEM}.(1) and \ref{prop:WLEMSP} \\
\hline
$\LEM{\n{\Sigma_k}} \to \LEM{\Pi_k}$ & $\HA + \DNE{\Sigma_{k-1}}$ & Corollary \ref{cor:NLEM1} \\
\hline
$\DNE{\Sigma_k} \to \LEM{\Delta_k}$ & $\HA$ & Fact \ref{fact:ABHK}.(4) \\
\hline
$\DNE{\Sigma_k} \to \DML{\Pi_k}$ & $\HA$ & Proposition \ref{prop:SDNE_DML} \\
\hline
$\DNE{(\Pi_k \lor \Pi_k)} \to \LEM{\Delta_k}$ & $\HA$ & Fact \ref{fact:ABHK}.(3) \\
\hline
$\DNE{(\Pi_k \lor \Pi_k)} \to \DML{\Pi_k}$ & $\HA$ & Corollary \ref{cor:MC} and Proposition \ref{prop:PSDML} \\
\hline
$\DNE{(\Pi_k \lor \Pi_k)} \to \DML{\Sigma_k}$ & $\HA$ & Corollary \ref{cor:MC} \\
\hline
$\DML{\Sigma_k} \to \DNE{(\Pi_k \lor \Pi_k)}$ & $\HA + \DNE{\Sigma_{k-1}}$ & Corollary \ref{cor:MC} \\
\hline
$\LEM{\n{\Sigma_k}} \to \DML{\Pi_k}$ & $\HA + \DNS{\Sigma_{k-1}}$ & Proposition \ref{prop:WLEMSP} and Corollary \ref{cor:nLEM_DML}.(1) \\
\hline
$\LEM{\n{\Sigma_k}} \to \DML{\Sigma_k}$ & $\HA$ & Corollary \ref{cor:nLEM_DML}.(1) \\
\hline
$\LEM{\Delta_k} \to \DNE{(\Delta_k \lor \Delta_k)}$ & $\HA$ & Corollary \ref{cor:DLEM} \\
\hline
$\LEM{\Delta_k} \to \LEM{\n{\Delta_k}}$ & $\HA$ & Proposition \ref{prop:NLEM}.(2) \\
\hline
$\LEM{\n{\Delta_k}} \to \LEM{\Delta_k}$ & $\HA + \DNE{\Sigma_{k-1}}$ & Proposition \ref{prop:NLEM}.(2) \\
\hline
$\DML{\Pi_k} \to \LEM{\n{\Delta_k}}$ & $\HA + \DNS{\Sigma_{k-1}}$ & Corollary \ref{cor:PDML_DLEM}.(1) \\
\hline
$\DML{\Sigma_k} \to \LEM{\n{\Delta_k}}$ & $\HA + \DNS{\Sigma_{k-1}}$ & Corollary \ref{cor:PDML_DLEM}.(1) \\
\hline
$\DNE{(\Delta_k \lor \Delta_k)} \to \LEM{\Sigma_{k-1}}$ & $\HA$ & Proposition \ref{prop:DNE3} \\
\hline
$\DNE{(\Delta_k \lor \Delta_k)} \to \DML{\n{\Delta_k}}$ & $\HA$ & Corollary \ref{cor:DNE2}.(6) \\
\hline
$\DML{\n{\Delta_k}} \to \DNE{(\Delta_k \lor \Delta_k)}$ & $\HA + \DNE{\Sigma_{k-1}}$ & Corollary \ref{cor:DNE2}.(6) \\
\hline
$\LEM{\n{\Delta_k}} \to \DML{\n{\Delta_k}}$ & $\HA$ & Corollary \ref{cor:nLEM_DML}.(2) \\
\hline
$\LEM{\n{\Delta_k}} \to \DML{\Delta_k}$ & $\HA$ & Corollary \ref{cor:nLEM_DML}.(2) \\
\hline
$\DML{\n{\Delta_k}} \to \LEM{\n{\Sigma_{k-1}}}$ & $\HA + \DNS{\Sigma_{k-2}}$ & Proposition \ref{prop:DDML1}.(2) \\
\hline
$\DML{\Delta_k} \to \DML{\n{\Delta_k}}$ & $\HA + \DNE{\Sigma_{k-1}}$ & Proposition \ref{prop:DDML3} \\
\hline
$\DML{\Delta_k} \to \LEM{\n{\Sigma_{k-1}}}$ & $\HA + \DNS{\Sigma_{k-2}}$ & Proposition \ref{prop:DDML1}.(1) \\
\hline
$\LEM{\n{\Sigma_{k-1}}} \to \LEM{\Sigma_{k-1}}$ & $\HA + \DNE{\Sigma_{k-1}}$ & Corollary \ref{cor:NLEM2} \\
\hline
\end{tabular}
\caption{Implications in Figure \ref{fig:DNS}}\label{table:implications}
\end{table}

As a summary, we illustrate the relationships between the principles we have dealt with so far.
However, the structure of such relationships is somewhat complicated. 
As we have shown, some minor differences in some of the principles are smoothed out in the theory $\HA + \DNS{\Sigma_{k-1}}$. 
Therefore, by illustrating the relationships between the principles in the theory $\HA + \DNS{\Sigma_{k-1}}$, one can grasp the structure in perspective.
In fact, in the presence of $\DNS{\Sigma_1}$ (in second-order arithmetic), a lot of equivalences in classical reverse mathematics can be established even intuitionistically (cf.~\cite[Proposition 1.1]{FuKo} and \cite[Theorem 2.10]{Fuji20}).

Figure \ref{fig:DNS} summarizes the derivability relation between several principles over $\HA + \DNS{\Sigma_{k-1}}$ with supplementary information about the situation over $\DNE{\Sigma_{k-1}}$. 
In fact, except $\LEM{\n{\Sigma_k}} \to \DML{\Pi_k}$, $\DML{\Sigma_k} \to \LEM{\n{\Delta_k}}$, $\DML{\Pi_k} \to \LEM{\n{\Delta_k}}$, $\DML{\Delta_k} \to \LEM{\n{\Sigma_{k-1}}}$ and $\DML{\n{\Delta_k}} \to \LEM{\n{\Sigma_{k-1}}}$, all the (non-dashed) implications presented in Figure \ref{fig:DNS} are provable even in $\HA$. 
However, one should note that the principle located at each vertex is one adequately selected from the equivalence class of principles modulo $\HA + \DNS{\Sigma_{k-1}}$, and hence, the $\HA$-provability depends on the choice of the representatives for the vertices.
For instance, we can replace $\LEM{\n{\Sigma_k}}$ with $\LEM{\n{\Pi_k}}$ by Proposition \ref{prop:WLEMSP}. 
Then $\LEM{\n{\Pi_k}} \to \DML{\Pi_k}$ is provable in $\HA$ while $\LEM{\n{\Pi_k}} \to \DML{\Sigma_k}$ is so in $\HA + \DNS{\Sigma_{k-1}}$.

As already mentioned so far, several underivability results are proved in the literature (cf.~\cite{ABHK,F20,FIN,FINSY,Kohl08,Kohl}). 
In particular, Fujiwara et al.~\cite{FINSY} recently introduced a fairy useful method to separate $\Sigma_k$ variants of the logical principles. 
All the underivability results in \cite{ABHK} obtained by using several kinds of functional interpretations can be proven uniformly in the methodology (see \cite[Example 10]{FINSY}). 
Furthermore, as in \cite[Section 4]{F20}, one can also prove $\LEM{\Sigma_{k-1}} \not \to \DML{\n{\Delta_k}}$, $\LEM{\Sigma_{k-1}} + \DML{\n{\Delta_k}} \not \to \DML{\Delta_k}$ and $\LEM{\Sigma_{k-1}} + \DML{\Delta_k} \not \to \LEM{\Delta_k}$ by this method. 
However, the separations of the principles which are equivalent only in the presence of $\DNE{\Sigma_{k-1}}$ (or even $\DNS{\Sigma_{k-1}}$) are extremely delicate. 
One needs further effort for such separations. 

In Section \ref{section:DML}, we investigated the principles which are closely related to the induction principle such as the contrapositive collection principle and the least number principle over $\HA$, which contains the full induction scheme, in order to examine the logical strength of them. 
Then we found that $\COLL{\Pi_k}$, $\LN{\Pi_k}$ and $\LN{\Sigma_k}$ are equivalent to $\DML{\Sigma_k} + \DNE{\Sigma_{k-1}}$, $\LEM{\Pi_k}$ and $\LEM{\Sigma_k}$ over $\HA$, respectively (see Theorem \ref{thm:LN} and Corollary \ref{cor:PCOLL}). 
On the other hand, it is interesting to analyze the relationship between these principles and the induction principle over intuitionistic arithmetic only with restricted induction scheme.

\clearpage

We close this paper with a list of principles which we have investigated. 

\begin{tabbing}
\hspace{30mm} \= \hspace{75mm} \= \hspace{50mm} \kill
$\LEM{\Gamma}$ \> $\varphi \lor \neg \varphi$ \> ($\varphi \in \Gamma$) \\
$\LEM{\Gamma}^\bot$ \> $\varphi \lor \varphi^\bot$ \> ($\varphi \in \Gamma$) \\
$\LEM{\Delta_k}$ \> $(\varphi \leftrightarrow \psi) \to \varphi \lor \neg \varphi$ \> ($\varphi \in \Sigma_k$ and $\psi \in \Pi_k$) \\
$\LEM{\Delta_k}^{\bot, \Sigma}$ \> $(\varphi \leftrightarrow \psi) \to \varphi \lor \varphi^\bot$ \> ($\varphi \in \Sigma_k$ and $\psi \in \Pi_k$) \\
$\LEM{\Delta_k}^{\bot, \Pi}$ \> $(\varphi \leftrightarrow \psi) \to \psi \lor \psi^\bot$ \> ($\varphi \in \Sigma_k$ and $\psi \in \Pi_k$) \\
$\LEM{\n{\Delta_k}}$ \> $(\varphi \leftrightarrow \psi) \to \neg \varphi \lor \neg \neg \varphi$ \> ($\varphi \in \Sigma_k$ and $\psi \in \Pi_k$) \\
$\DNE{\Gamma}$ \> $\neg \neg \varphi \to \varphi$ \> ($\varphi \in \Gamma$) \\
$\PEIRCE{\Gamma}$ \> $((\varphi \to \psi) \to \varphi) \to \varphi$ \> ($\varphi \in \Gamma$ and $\psi$ is any formula) \\
$\DNS{\Gamma}$ \> $\forall x\, \neg \neg \varphi(x) \to \neg \neg\, \forall x\varphi(x)$ \> ($\varphi(x) \in \Gamma$) \\
$\DML{\Gamma}$ \> $\neg (\varphi \land \psi) \to \neg \varphi \lor \neg \psi$ \> ($\varphi, \psi \in \Gamma$) \\
$\DML{(\Gamma, \Theta)}$ \> $\neg(\varphi \land \psi) \to \neg \varphi \lor \neg \psi$ \> ($\varphi \in \Gamma$ and $\psi \in \Theta$) \\
$\DML{\Delta_k}$ \> $(\varphi \leftrightarrow \varphi') \land (\psi \leftrightarrow \psi')$ \>  \\
 \> \hspace{0.5in} $\to (\neg (\varphi \land \psi) \to \neg \varphi \lor \neg \psi)$ \> ($\varphi, \psi \in \Sigma_k$ and $\varphi', \psi' \in \Pi_k$) \\
 $\DML{\n{\Delta_k}}$ \> $(\varphi \leftrightarrow \varphi') \land (\psi \leftrightarrow \psi')$ \>  \\
 \> \hspace{0.5in} $\to (\neg (\neg \varphi \land  \neg  \psi) \to \neg \neg  \varphi \lor \neg \neg  \psi)$ \> ($\varphi, \psi \in \Sigma_k$ and $\varphi', \psi' \in \Pi_k$) \\
$\DML{(\Delta_k, \Theta)}$ \> $(\varphi \leftrightarrow \varphi') \to (\neg(\varphi \land \psi) \to \neg \varphi \lor \neg \psi)$ \> ($\varphi \in \Sigma_k$, $\varphi' \in \Pi_k$ and $\psi \in \Theta$) \\
$\DML{\Gamma}^\bot$ \> $\neg(\varphi \land \psi) \to \varphi^\bot \lor \psi^\bot$ \> ($\varphi, \psi \in \Gamma$) \\
$\DML{(\Gamma, \Theta)}^\bot$ \> $\neg(\varphi \land \psi) \to \varphi^\bot \lor \psi^\bot$ \> ($\varphi \in \Gamma$ and $\psi \in \Theta$) \\
$\DML{\Delta_k}^\bot$ \> $(\varphi \leftrightarrow \varphi') \land (\psi \leftrightarrow \psi')$ \>  \\
 \> \hspace{0.5in} $\to (\neg (\varphi \land \psi) \to \varphi^\bot \lor \psi^\bot)$ \> ($\varphi, \psi \in \Sigma_k$ and $\varphi', \psi' \in \Pi_k$) \\
$\DML{(\Delta_k, \Gamma)}^{\bot, \Sigma}$ \> $(\varphi \leftrightarrow \varphi') \to (\neg(\varphi \land \psi) \to \varphi^\bot \lor \psi^\bot)$ \> ($\varphi \in \Sigma_k$, $\varphi' \in \Pi_k$ and $\psi \in \Gamma$) \\
$\DML{(\Delta_k, \Gamma)}^{\bot, \Pi}$ \> $(\varphi \leftrightarrow \varphi') \to (\neg(\varphi \land \psi) \to (\varphi')^\bot \lor \psi^\bot)$ \> ($\varphi \in \Sigma_k$, $\varphi' \in \Pi_k$ and $\psi \in \Gamma$) \\
$\CD{(\Gamma, \Theta)}$ \> $\forall x (\varphi \lor \psi(x)) \to \varphi \lor \forall x \psi(x)$ \> ($\varphi \in \Gamma$, $\psi(x) \in \Theta$ and $x \notin \FV(\varphi)$) \\
$\DUAL{\Gamma}$ \> $\neg \varphi \to \varphi^\bot$ \> ($\varphi \in \Gamma$) \\
$\DUAL{\Delta_k}^\Sigma$ \> $(\varphi \leftrightarrow \psi) \to (\neg \varphi \to \varphi^\bot)$ \> ($\varphi \in \Sigma_k$ and $\psi \in \Pi_k$) \\
$\DUAL{\Delta_k}^\Pi$ \> $(\varphi \leftrightarrow \psi) \to (\neg \psi \to \psi^\bot)$ \> ($\varphi \in \Sigma_k$ and $\psi \in \Pi_k$) \\
$\WDUAL{\Gamma}$ \> $\neg \varphi^\bot \to \neg \neg \varphi$ \> ($\varphi \in \Gamma$) \\
$\COLL{\Gamma}$ \> $\forall w \, \exists y < x \, \forall z < w \, \varphi(y, z) \to \exists y < x \, \forall z \, \varphi(y, z)$ \> ($\varphi(y, z) \in \Gamma$) \\
$\LN{\Gamma}$ \> $\exists x \varphi(x) \to \exists x (\varphi(x) \land \forall y < x\, \neg \varphi(y))$ \> ($\varphi \in \Gamma$) 

\end{tabbing}

\section*{Acknowledgement}
The authors thank to the anonymous referee for the valuable and insightful comments.
The first author was supported by JSPS KAKENHI Grant Numbers JP19J01239 and JP20K14354, and the second author by JP19K14586.

\bibliographystyle{plain}
\bibliography{references}

\end{document}